\documentclass[]{article}
\pdfoutput=1
\usepackage{amsmath}
\usepackage[scale=0.7]{geometry}
\usepackage{cases}
\usepackage{amssymb,amsfonts,amsthm}
\usepackage{bm}
\usepackage{bigstrut,multirow,rotating}
\usepackage{mathrsfs}
\usepackage{subfigure}
\makeatletter

\newcommand{\Rmnum}[1]{\expandafter\@slowromancap\romannumeral #1@}

\newcommand{\black}{\textcolor{black}}
\makeatother
\usepackage{authblk}
\usepackage{hyperref}                                                  
\hypersetup{hypertex = true, colorlinks = true, linkcolor = blue, anchorcolor = blue, citecolor =blue}
\usepackage[square,numbers]{natbib}
\newtheorem{theorem}{Theorem}
\newtheorem{proposition}{Proposition}
\newtheorem{lemma}{Lemma}[section]
\newtheorem{definition}{Definition}
\newtheorem{assumption}{Assumption}
\newenvironment{shrinkeq}[1]
{\bgroup
	\addtolength\abovedisplayshortskip{#1}
	\addtolength\abovedisplayskip{#1}
	\addtolength\belowdisplayshortskip{#1}
	\addtolength\belowdisplayskip{#1}}
{\egroup\ignorespacesafterend}
\usepackage{lipsum}
\usepackage{amsfonts}
\usepackage{graphicx}
\usepackage{epstopdf}
\usepackage{algorithm}
\usepackage{algorithmic}
\usepackage{color}
\title{Unbiased Markov chain quasi-Monte Carlo for Gibbs samplers}

\author[1]{Jiarui Du}
\author[2]{Zhijian He \thanks{Corresponding author: hezhijian@scut.edu.cn}}

\affil[1,2]{School of Mathematics, South China University of Technology, Guangzhou 510641, Guangdong, People’s Republic of China}
\begin{document}
	
	\maketitle
	
	\begin{abstract}
	In statistical analysis, Monte Carlo (MC) stands as a classical numerical integration method. When encountering challenging sample problem, Markov chain Monte Carlo (MCMC) is a commonly employed method. However, the MCMC estimator is biased after a fixed number of iterations. Unbiased MCMC, an advancement achieved through coupling techniques, addresses this bias issue in MCMC. It allows us to run many short chains in parallel. Quasi-Monte Carlo (QMC), known for its high order of convergence, is an alternative of MC. By incorporating the idea of QMC into MCMC, Markov chain quasi-Monte Carlo (MCQMC) effectively reduces the variance of MCMC, especially in Gibbs samplers. This work presents a novel approach that integrates unbiased MCMC with MCQMC, called as an unbiased MCQMC method. This method renders unbiased estimators while improving the rate of convergence significantly. Numerical experiments demonstrate that for Gibbs sampling, unbiased MCQMC with a sample size of $N$ yields a faster root mean square error (RMSE) rate than the  \(O(N^{-1/2})\) rate of unbiased MCMC, toward an RMSE rate of \(O(N^{-1})\) for low-dimensional problems. Surprisingly, in a challenging problem of 1049-dimensional P\'olya Gamma Gibbs sampler, the RMSE can still be reduced by several times for moderate sample sizes. In the setting of parallelization, unbiased MCQMC also performs better than unbiased MCMC, even running with short chains. 
	\end{abstract}
	
	
	\section{Introduction}\label{intro}
	As widely known, Monte Carlo (MC) method is utilized to estimate the expectation $E_\pi[f(\black{\bm{X}})]$ for a target distribution $\pi$ with respect to a certain function $f$. Over the past few decades, MC has been widely applied in diverse fields such as science, engineering, finance, industry, and statistical inference \citep{glasserman2004,owenmc}. In practical applications, MC requires sampling from the target distribution. In cases where direct sampling from the target distribution is not feasible, such as the posterior distribution in Bayesian computation, one often resorts to Markov chain Monte Carlo (MCMC) method. In MCMC, a Markov chain is simulated, and sample averages are used to estimate the expectation of the target distribution. Classical MCMC algorithms include Metropolis-Hastings (M-H) samplers and Gibbs samplers. For \black{Harris recurrent} MCMC algorithm, as the number of iterations increases, sample averages converge to the expectation with probability 1 (w.p.1), i.e.,
	$$
	\frac{1}{N}\sum\limits_{i=1}^{N}f(\black{{\bm{X}_i}}) \stackrel{w.p.1}{\to} E_\pi[f(\black{\bm{X}})],
	$$
	where $N$ is the sample size, and $(\black{{\bm{X}_t}})_{t\ge0}$ is a Markov chain with $\pi$ as its stationary distribution. This ensures the consistency of MCMC algorithms  \citep{robert2004}. 
	
	However, for a fixed number of iterations, if the Markov chain is initialized from a state outside the stationary distribution, it is uncertain when the chain will enter the stationary distribution. Consequently, the obtained samples always contain a portion that does not follow the stationary distribution, leading to a bias known as ``burn-in bias". \black{This bias presents challenges for MCMC parallelization and has prompted the development of various methods, such as the circularly-coupled Markov chain proposed by Neal \cite{neal2017}.} Recently, Glynn and Rhee \cite{glynn2014} utilized coupling techniques to entirely eliminate the mean bias of the Markov chain traversal average as represented by iterative stochastic functions. Jacob et al. \cite{jacob2019} eliminated the bias of conditional particle filters by using coupling chains. Jacob et al. \cite{jacob2020} then proposed an unbiased MCMC method, which results in unbiased estimators for M-H samplers and Gibbs samplers. \black{This method has also been widely applied in other fields, such as Hamiltonian Monte Carlo \citep{heng2019} and Bayesian inference with intractable likelihoods \citep{middleton2020}.}
	
	Another limitation of MC in practical applications is due to its slow root mean square error (RMSE) rate $O(N^{-1/2}$) for $N$ samples. This rate also holds for MCMC and unbiased MCMC. To improve MC, one may use deterministic low-discrepancy sequences instead of independent and identically distributed (IID) random sequences, known as quasi-Monte Carlo (QMC) methods \citep{bookdick2010,booknieder1992}.  For a $d$-dimensional integral, QMC yields deterministic error bounds of $O(N^{-1}(\log N)^d)$ for certain regular functions, which are asymptotically superior to MC rate $O(N^{-1/2})$. Given the advantages of QMC over MC, it is worthwhile to consider the idea of replacing IID sequences in MCMC processes with deterministic sequences. However, directly substituting typical low-discrepancy sequences with IID sequences may not produce correct results. This is due to the fact that deterministic low-discrepancy sequences exhibit strong correlations, which disrupt the Markov property of the Markov chain in the MCMC process, as discussed in  \citep[Section 3.2]{tribble}. Therefore, another type of uniform sequences is required. Motivated by Niederreiter \cite{niederreiter1986}, Owen and Tribble \cite{owen2005, tribble2008} proposed Markov chain quasi-Monte Carlo (MCQMC) method. They have proved that using completely uniformly distributed (CUD) or weakly CUD (WCUD) sequences in discrete spaces can yield consistent results. Chen et al. \cite{chen2011} extended these consistency results to continuous spaces. For properties related to (W)CUD sequences, we refer to \citep{niederreiter1986,tribble2008}.

	Currently, the theoretical foundations for determining the convergence rate of MCQMC are weak. \black{If the update function is a strong contractive mapping, along with other conditions}, Chen \cite{chen} proved that for any \(\epsilon > 0\), MCQMC exhibits a convergence order similar to QMC methods, i.e., \(O(N^{-1+\epsilon})\). However, this result does not represent a typical absolute error for all samples mean and the required sequence is related to the strong contracting condition. Liu \cite{liu2023} applied Chen's convergence theory to Langevin diffusion models. Dick et al. \cite{dick2014,dick2016} presented a Koksma-Hlawka inequality for MCQMC. \black{Under the assumption that the transition kernel is uniformly ergodic, along with other conditions}, they demonstrated the existence of a driving sequence achieving a convergence rate for MCQMC of \(O(N^{-1/2}(\log N)^{1/2})\). Furthermore, if the update function satisfies the ``anywhere-to-anywhere" condition, there exists a driving sequence that results in a convergence rate of $O(N^{-1+\epsilon})$. However, this result is merely an existence result.
	
	\black{Although there is currently no theoretical framework in the literature to explain the high order convergence rate of (W)CUD sequences, empirical results have confirmed their satisfactory performance,}
	especially in Gibbs samplers. For example, it was observed in \citep{harase2021} that Harase's method yields several orders of magnitude reduction in variance. These findings motivate us to integrate unbiased MCMC with MCQMC, with the goal of preserving the unbiasedness of MCMC estimators while simultaneously reducing the variance of the estimators. To this end, we propose a method of unbiased MCQMC for Gibbs samplers. The unbiased MCQMC method is beyond a mere amalgamation of unbiased MCMC and MCQMC. The challenges involved include determining how to effectively utilize array-(W)CUD sequences in unbiased MCMC to enhance the performance of MCQMC. Moreover, we should note that the chains generated with array-(W)CUD sequences are no longer Markovian.
	
	The contribution of this paper is three-fold. Firstly, we introduce the unbiased MCQMC method by integrating unbiased MCMC and MCQMC and provide a strategy for utilizing array-(W)CUD sequences in unbiasd MCQMC. This method only requires the use of array-(W)CUD sequences in the partial sampling process of unbiased MCMC. Secondly, we prove the unbiasedness of the proposed methods and theoretically establish the upper bounds on the second moment of the bias term in unbiased MCQMC without relying on the Markov property, which was required in \citep{jacob2020}. \black{Lastly, numerical experiments across various dimensions demonstrate that the RMSE of unbiased MCQMC is significantly reduced compared to that of unbiased MCMC and exhibits a faster convergence rate than $O(N^{-1/2})$.}
	
	The paper is organized as follows. Section \ref{sec:back} provides a brief introduction to QMC, MCMC, MCQMC, and unbiased MCMC. Section \ref{sec:unmcqmc} presents a comprehensive framework for the unbiased MCQMC method. Section \ref{sec:theory} establishes the unbiasedness of the estimator in unbiased MCQMC, accompanied by a discussion of the estimator's variance. Section~\ref{sec:experiments} provides numerical experiments for different dimensional Gibbs samplers to show the  the effectiveness of unbiased MCQMC. The conclusions and discussions are given in Section \ref{sec:conclusions}. Scripts of MATLAB (MATLAB R2023b) are available from \url{https://github.com/Jiarui-Du/ubmcqmc}.
	
	\section{Background}
	\label{sec:back}
	\subsection{Quasi-Monte Carlo}
	QMC is a deterministic version of MC, known for its higher order rate of convergence compared to MC. Here, we provide a brief introduction to QMC, and further information can be found in the monographs \cite{bookdick2010, booknieder1992}. 
	
	QMC is commonly used for the numerical integration of  $d$-dimensional integral $\mu = \int_{(0,1)^d} f(\black{\bm{u}}) \mathrm{d} \black{\bm{u}}$. The estimator is \(\black{\hat{\mu}} = \frac{1}{N} \sum_{i=1}^N f(\bm{u}_{i})\), where $\bm{u}_{i}\in (0,1)^d$. In comparison to MC, where $\bm{u}_{i}$ are IID, $\{\bm{u}_{i}\}_{i=1}^N$ used in QMC are low-discrepancy points. The discrepancy of a point set is typically quantified using the star discrepancy, which is defined by
	\begin{equation}\label{eq:stardis}
		D_N^*\left(\bm{u}_{1}, \ldots, \bm{u}_{N}\right) = \sup_{\bm{z} \in (0,1)^d} \left|\frac{1}{N} \sum_{i=1}^N 1_{\left\{\bm{u}_{i} \in [\mathbf{0}, \black{\bm{z}})\right\}} - \prod_{j=1}^d z_j\right|.
	\end{equation}
	A point set with a star discrepancy of $O(N^{-1}(\log N)^{d})$ is referred to as a low-discrepancy point set. Commonly used QMC point sets include Sobol' sequences, Faure sequences, Halton sequences, lattice rules, and others. The star discrepancy plays a crucial role in determining the estimation error of integrals over $(0,1)^d$. The classical Koksma-Hlawka inequality provides an error bound, i.e.,
	\begin{align}\label{eq:KH_inq}
		|\black{\hat{\mu}}-\black{\mu}|\leq D_N^*\left(\bm{u}_{1}, \ldots, \bm{u}_{N}\right) V_{\mathrm{HK}}(f),
	\end{align}
	where $V_{\mathrm{HK}}(\cdot)$ is the variation (in the sense of Hardy and Krause) of a function. For the definition and properties of $V_{\mathrm{HK}}(\cdot)$, we refer to \citep{owen2005hk}. To facilitate error estimation, randomized QMC (RQMC) is commonly used. Typical randomization methods include random shifts \citep{cranley1976}, digital shifts \citep{owenqmc} and scrambling \citep{owen1995}.
	
	\subsection{Markov chain Monte Carlo}\label{mcmc}
	MCMC extends the scope of MC. In this section, following \citep[Section 2.4]{chen2011}, we shall briefly introduce MCMC by update functions. Given the target distribution $\pi$ on $\Omega \subseteq \mathbb{R}^d$ and a measurable real function $f$ which is integrable with respect to $\pi$, our goal is to sample $\black{\bm{X}} \sim \pi$ and approximate $\mu(f) = \int f(\black{\bm{x}})\pi(\black{\bm{x}})d\black{\bm{x}}$. Consider a Markov chain $(\black{{\bm{X}_t}})_{t\ge0}$ with $P$ as its transition kernel and an update function $\phi$ such that $\pi$ is a stationary distribution for $(\black{{\bm{X}_t}})_{t\ge0}$ and $\pi_0$ as some initial distribution on $\Omega$. We start by $\black{{\bm{X}_0}} \sim \pi_0$, and then, for $i\ge 1$, the state updating process is defined as $\black{{\bm{X}_i}} = \phi(\black{{\bm{X}_{i-1}}},\bm{u}_{i}) \sim  P(\black{{\bm{X}_{i-1}}},\cdot)$, where $\bm{u}_{i} \sim \mathbf{U}(0,1)^d$. \black{Here, the dimension of \(\bm{u}_{i}\) does not necessarily have to be \(d\), but for simplicity, we keep it consistent with the dimension of the states \(\bm{X}_i\)}. Then we estimate $\mu(f)$ by $\frac{1}{N}\sum_{i=1}^{N}f(\black{{\bm{X}_i}})$. 
	The update functions for the Gibbs sampler are shown in the following.
	
	\textbf{Sequential Scan Gibbs Update.}  \black{Let the current state be $\bm{X} = ({\bm{x}_1},\ldots,{\bm{x}_s}) \in \Omega\subset \mathbb{R}^d$ with ${\bm{x}_i} \in \mathbb{R}^{d_i}$ and $d = \sum_{i=1}^s d_i$. Let $\bm x_{-i}$ be all  $\bm x_j$ with $j\neq i$.  Denote $\psi_i({{\bm{x}_{-i}}},\cdot) : (0,1)^{d_i} \to \mathbb{R}^{d_i}$ as a $d_i$-dimensional generator of the full conditional distribution of ${\bm{x}_i}$ given  $\bm x_{-i}$. For $\bm{u} = ({\bm{v}_1},\ldots,{\bm{v}_s}) \in (0,1)^d$ with ${\bm{v}_i} \in (0,1)^{d_i}$,  the next state can be updated via the update function 
		\begin{align}\label{gibbs}
			\bm{Y}= ({\bm{y}_1},\ldots,{\bm{y}_s}) = \phi\left(\bm{X},\bm{u}\right) ,
		\end{align}
		where $\bm{y}_i = \psi_i(\bm{x}_{[i]},\bm{v}_i)$ and $\bm{x}_{[i]} = \left(\bm{y}_1, \ldots, \bm{y}_{i-1}, \bm{x}_{i+1}, \ldots,\bm{x}_{s}\right), 1\le i \le s$.}
	
	To formulate the update function \eqref{gibbs}, one needs to transform the process of state transition into $s$ generator functions of uniform random variables. There are several ways to generate non-uniform random variables using uniform random variables \cite{devroye1986,owenmc}. 
	A commonly used method is the Rosenblatt–Chentsov transformation \cite{rosenblatt:1952}. Consider a  generic  random vector $\bm x\in\mathbb{R}^d$ with a joint CDF $F(\bm x)$. Let $F_1(x_1)$ be the marginal CDF of the first component $x_1$ and for $j = 2,\ldots,d$, let $F_j(\cdot|x_1,\ldots,x_{j-1})$ be the conditional CDF of $x_j$ given $x_1,\ldots,x_{j-1}$. To sample $\bm x\sim F$, the Rosenblatt–Chentsov transformation takes   
	$$
	x_1 = F_1^{-1}(u_1),~~\text{and}~~x_j = F_j^{-1}(u_j|x_1,\ldots,x_{j-1}),~~j= 2,\ldots,d,
	$$
	where $\bm{u} \sim \mathbf{U}(0,1)^d$. If the components of $\bm x$ are independent, it turns out to  take the inversion of the marginal CDFs.
	
	However, for many distributions including the standard normal distribution, the inverse CDF does not have a closed form. Fortunately, some remarkable numerical algorithms have been proposed to approximate inverse CDFs of common distributions. For example, the inverse CDF of the standard normal distribution can be approximated by high-precision numerical algorithms \cite{owenmc}. On the other hand, one can use numerical root-finding methods, such as Newton's method, to obtain inverse CDFs. In Section \ref{subsec:logistic}, we study the impact of different ways to formulating \eqref{gibbs} on the performance of unbiased MCQMC.
	
	\subsection{Markov chain quasi-Monte Carlo}\label{mcqmc}
	Replacing IID $\black{\mathbf{U}}(0,1)$ points with carefully designed sequences in MCMC is expected to yield more accurate estimates. In this section, we review MCQMC by following the introduction in \citep[Chapter 3]{tribble}. 
	
	Suppose we want to collect $N$ samples, each requiring $d$ uniform variables. The required uniform variables can be organized into a variable matrix as follows,
	\begin{equation}\label{eq:vm}
		\left[\begin{array}{cccc}
			v_1 & v_2 & \cdots & v_{d} \\
			v_{d+1} & v_{d+2} & \cdots & v_{2d} \\
			\vdots & \vdots & \ddots & \vdots \\
			v_{(N-1)d+1} & v_{(N-1) d+2} & \cdots & v_{Nd}
		\end{array}\right],
	\end{equation}
	where $v_i\stackrel{iid}\sim \black{\mathbf{U}}(0,1)$. The sequence $(v_1, v_2, \ldots, v_{Nd})$ is referred to as the ``driving sequence" of MCMC. The driving sequence should be designed carefully with certain uniformity in MCQMC. In this paper, (W)CUD sequences defined below are used as driving sequences.
	
	\begin{definition}
		A \black{deterministic} infinite sequence $(v_i)_{i\ge 1}$ in $(0,1)$ is called as a CUD sequence if for any $d \ge 1$, 
		$$
		D_N^*\left((v_{1},\ldots,v_{d}),(v_{2},\ldots,v_{d+1}), \ldots, (v_{N},\ldots,v_{N+d})\right) \to 0, \text{~as~} N\to\infty,
		$$
		where $D_N^*$ is defined in \eqref{eq:stardis}. A \black{random} infinite sequence $(v_i)_{i\ge 1}$ in $(0,1)$ is called as a WCUD sequence if for any $\epsilon>0$ and $d \ge 1$,
		$$
		\mathbb{P}\left(D_N^*\left((v_{1},\ldots,v_{d}),(v_{2},\ldots,v_{d+1}), \ldots, (v_{N},\ldots,v_{N+d})\right) > \epsilon\right) \to 0, \text{~as~} N\to\infty.
		$$    
	\end{definition}
	
	Theorem 3 in \cite{owen2005} stated that using (W)CUD sequences in discrete spaces can yield (weakly) consistent results under two technical conditions. First, the pre-images of the update function $\phi$ in $[0,1]^d$ corresponding to transitions from one state to another state must be Jordan measurable. This is  referred to as the regular proposal condition. Second, the update function  must be chosen such that weak consistency holds when using IID driving sequence. Chen et al. \cite{chen2011} demonstrated (weak) consistency results for the use of (W)CUD sequences in continuous spaces. The second condition in Theorem 3 in \cite{owen2005} is also required in continuous spaces. Additionally, for M-H sampler, the Jordan measurability of pre-images in multistage transitions is also required, while Gibbs sampler requires a form of contraction mapping. See Theorems 2 and 3 in \cite{chen2011} for details.
	
	In practice, we use finite (W)CUD sequences. So the triangular array (W)CUD sequences is more useful. A triangular array $(v_{N,1},v_{N,2},\ldots,v_{N,N})$ for $N$ in an infinite positive integers set $\mathscr{N}$ is called as array-CUD, if for any $d \ge 1$, as $N\to \infty$ through the values in $\mathscr{N}$, 
	\begin{equation}\label{eq:arrcud}
		D_{N-d+1}^*\left((v_{N,1},\ldots,v_{N,d}),(v_{N,2},\ldots,v_{N,d+1}), \ldots, (v_{N,N-d+1},\ldots,v_{N,N})\right)\to 0.
	\end{equation} 
	Array-WCUD is defined analogously. 
	
	Currently, several methods are recognized for constructing array-(W)CUD sequences. These include Liao's  method \cite{liao1998}, the Multiplicative Congruential Generator (MCG) method \cite{niederreiter1977,owen2005,tribble2008}, and the commonly used Linear Feedback Shift Register (LFSR) generator, also called Tausworthe generators \cite{chen2012,harase2021,harase2024,tribble}. In this paper, we refer to the generator in \citep{harase2021} as Harase's method and the generator in \citep{liao1998} as Liao's method. In the existing experimental results, Harase's method performs well on sample sizes $N \in \{2^{10},2^{11},\ldots,2^{32}\}$. However, it is limited to these specific sizes, whereas Liao's method does not have any specific requirements on the sample size. We will employ both methods in our experiments. We defer the detailed introduction of the two methods to Appendix \ref{app:method}.
	
	\subsection{Unbiased Markov chain Monte Carlo}
	By using coupling techniques, Jacob et al. \cite{jacob2020} proposed the unbiased MCMC method to addresses this bias issue in MCMC. In this section, we follow \citep{jacob2020} to construct an unbiased estimator for $E_\pi[f(\black{\bm{X}})]$. Let $P$ be the Markov transition kernel on $\Omega$ with $\pi$ as its stationary distribution, and let $\bar{P}$ be a transition kernel on the joint space $\Omega \times \Omega$. 
	Consider two Markov chains \((\black{{\bm{X}_t}})_{t\ge0}\) and \((\black{{\bm{Y}_t}})_{t\ge0}\) with \(\pi_0\) as their initial distribution on $\Omega$. Initially, two independent samples \(\black{{\bm{X}_0}}\) and \(\black{{\bm{Y}_0}}\) are drawn from \(\pi_0\). Following this, given \(\black{{\bm{X}_0}}\), the sample \(\black{{\bm{X}_1}}\) is drawn from \(P\). Then, for \(t\ge1\), given the pair \((\black{{\bm{X}_t}},\black{{\bm{Y}_{t-1}}})\), the pair \((\black{{\bm{X}_{t+1}}},\black{{\bm{Y}_{t}}})\) is drawn from \(\bar{P}\).
	Assume that \(\bar{P}\) is a coupling kernel for \((\black{{\bm{X}_t}})_{t\ge0}\) and \((\black{{\bm{Y}_t}})_{t\ge0}\), satisfying \(\black{{\bm{X}_{t+1}}}|(\black{{\bm{X}_t}},\black{{\bm{Y}_{t-1}}}) \sim P(\black{{\bm{X}_t}},\cdot)\) and \(\black{{\bm{Y}_{t}}}|(\black{{\bm{X}_{t}}},\black{{\bm{Y}_{t-1}}}) \sim P(\black{{\bm{Y}_{t-1}}},\cdot)\). Thus, for any \(t\ge0\), \(\black{{\bm{X}_t}}\) and \(\black{{\bm{Y}_t}}\) follow identical distributions. Moreover, the design of the coupling kernel \(\bar{P}\) ensures the existence of a random variable referred to as the meeting time \(\tau\), such that for any \(t\ge\tau\), \(\black{{\bm{X}_t}} = \black{{\bm{Y}_{t-1}}}\). \black{Assume that as $t \to \infty$, $E[f({\bm{X}_t})] \to E_\pi[f(\bm{X})]$}. For a given integer \(k\ge0\), we have
	\begin{shrinkeq}{-1ex}
		\begin{equation}\label{eq:ub}
			\begin{aligned}
				E_\pi[f(\black{\bm{X}})] = \lim\limits_{t\to\infty}E[f(\black{{\bm{X}_t}})] 
				&= E[f(\black{{\bm{X}_k}})]+\sum_{l=k+1}^\infty E[f(\black{{\bm{X}_l}})]-E[f(\black{{\bm{X}_{l-1}}})]\\
				&= E\left[f(\black{{\bm{X}_k}})+\sum_{l=k+1}^\infty \{f(\black{{\bm{X}_l}})-f(\black{{\bm{Y}_{l-1}}})\}\right]\\ 
				&= E\left[f(\black{{\bm{X}_k}})+\sum_{l=k+1}^{\tau-1} \{f(\black{{\bm{X}_l}})-f(\black{{\bm{Y}_{l-1}}})\}\right]. 
			\end{aligned}
		\end{equation}
	\end{shrinkeq}
	If $\tau-1<k+1$, $\sum_{l=k+1}^{\tau-1}\{f(\black{{\bm{X}_l}})-f(\black{{\bm{Y}_{l-1}}})\}$ is zero by convention. 
	It is worth noting that  \eqref{eq:ub} is not a formal derivation. Jacob et al. \cite{jacob2020} have proven the unbiasedness of the following estimator based on Assumptions \ref{ass1}-\ref{ass3}, without requiring the interchangeability of expectation and infinite sum,
	\begin{shrinkeq}{-1ex}
		\begin{align*}
			F_k(\black{\mathbf{X}},\black{\mathbf{Y}}):=f(\black{{\bm{X}_k}})+\sum_{l=k+1}^{\tau-1}\{f(\black{{\bm{X}_l}})-f(\black{{\bm{Y}_{l-1}}})\},
		\end{align*}
	\end{shrinkeq}
	where $\black{\mathbf{X}}$ and $\black{\mathbf{Y}}$ represent two coupled Markov chains \((\black{{\bm{X}_t}})_{t\ge0}\) and \((\black{{\bm{Y}_t}})_{t\ge0}\). Furthermore, \black{for a given integer \(m\geq k\), we can run $\max(m,\tau)$ iterations, representing the length of the coupled chains}, and compute the unbiased estimator \(F_l(\black{\mathbf{X}},\black{\mathbf{Y}})\) where \(l=k,k+1,\ldots,m\). Then by taking the average of these \(m-k+1\) unbiased estimators $F_{k:m}(\black{\mathbf{X}},\black{\mathbf{Y}}) = \sum_{l=k}^{m}F_l(\black{\mathbf{X}},\black{\mathbf{Y}})/(m-k+1)$, we obtain the time-averaged estimator as follows,
	\begin{shrinkeq}{-1ex}
		\begin{equation}\label{eq:Fkm}
			\begin{aligned}
				F_{k:m}(\black{\mathbf{X}},\black{\mathbf{Y}}) 
				&=\frac{1}{m-k+1}\sum_{l=k}^{m}f(\black{{\bm{X}_l}})+\sum_{l=k+1}^{\tau-1}\min\!\left(1,\frac{l-k}{m-k+1}\right)\{f(\black{{\bm{X}_l}})-f(\black{{\bm{Y}_{l-1}}})\}\\
				&=:\text{MCMC}_{k:m}+\text{BC}_{k:m}.
			\end{aligned}	
		\end{equation}
	\end{shrinkeq}
	When considering \((\black{{\bm{X}_t}})_{t\ge0}\) only, it is a standard Markov chain. The first term $\text{MCMC}_{k:m}$ in \eqref{eq:Fkm} can be interpreted as the average value of the standard Markov chain after \(m\) iterations, with a burn-in period of \(k\). Clearly, it is biased. Thus, the second term $\text{BC}_{k:m}$ in \eqref{eq:Fkm} can be considered as a bias correction term. 
	To obtain the unbiased estimator \eqref{eq:Fkm}, it is necessary to know that how to construct the specific coupling kernel \(\bar{P}\) and how to sample from \(\bar{P}\). Following \citep{jacob2020}, we consider the maximal coupling method. For other coupling methods, we refer to \citep[Section 4]{jacob2020}. The maximal coupling of distributions \(p\) and \(q\) on the space \(\Omega\) is the joint distribution of the random variables \((Z, Z^{\prime})\). This distribution satisfies the marginal distributions \(Z \sim p\) and \(Z^{\prime} \sim q\), and it maximizes the probability of the event \(\{Z=Z^{\prime}\}\). 
	
	By incorporating the array-(W)CUD sequences in unbiased MCMC, we have successfully accelerated the convergence speed of unbiased MCMC with MCQMC. Next, we provide the construction of the unbiased MCQMC method in detail, along with some theoretical results.
	
	\section{Unbiased Markov chain quasi-Monte Carlo}\label{sec:unmcqmc}
	In unbiased MCMC, we construct two Markov chains, while MCQMC deals with only one chain. When applying MCQMC to unbiased MCMC, a fundamental question arises. Which chain's IID sequence should be replaced with a (W)CUD sequence --- only dealing with \black{chain $\mathbf{X}$, only dealing with chain $\mathbf{Y}$, or both chains $\mathbf{X}$ and $\mathbf{Y}$}?  Combined with MCQMC, we describe the sample process of unbiased MCMC as driven by two variable matrices. For easy of notation, we rewrite the sample process as follows, which is a combination of Algorithms 1 and 2 for Gibbs sampler in \citep{jacob2020}. 
	
	\begin{algorithm}
		\caption{Unbiased estimator $F_{k:m}(\mathbf{X},\mathbf{Y})$ by IID sequences for Gibbs sampler}
		\label{alg:ubest_rand}
		\begin{algorithmic}[1]
			\STATE{Given integers $k$, $m\ge k$, initial distribution $\pi_0$, update function $\phi$ and generator function $\psi_i$ defined in \eqref{gibbs}, conditional density $P_i,i=1,2,\ldots,s$ of the components of each part with $d = \sum_{i=1}^s d_i$.} 
			\STATE{Sample ${\bm{X}_0},{\bm{Y}_0}\sim \pi_0$, and $\bm{V}_1 \sim \mathbf{U}(0,1)^d$, set ${\bm{X}_{1}} = \phi({\bm{X}_0},\bm{V}_1)$, $t=1,\tau=\infty$}
			\STATE{Repeat steps 4-14 until $t \ge \max(m,\tau)$}
			\IF{$t < \tau$}
			\STATE{\black{Sample $\bm{V}_{t+1} \sim \mathbf{U}(0,1)^d$ and set $({\bm{v}_1}, \ldots,{\bm{v}_s}) = {\bm{V}_{t+1}}$}}
			\STATE{\black{Set $({\bm{x}_1}, \ldots,{\bm{x}_s}) = {\bm{X}_t}, ({\bm{y}_1}, \ldots,{\bm{y}_s}) = \bm{Y}_{t-1}$}}
			\FOR{$i$ in $1$ to $s$} 
			\STATE{\black{Set $p_i(\cdot) = P_i({\bm{x}_{-i}},\cdot),q_i(\cdot) = P_i({\bm{y}_{-i}},\cdot)$}}
			\STATE{\black{Sample $w \sim \mathbf{U}(0,1)$, set ${\bm{x}_{i}} = \psi_i({\bm{x}_{-i}},\bm{v}_i)$, if $p_i({\bm{x}_{i}})w \le q_i({\bm{x}_{i}})$, set ${\bm{y}_{i}} = {\bm{x}_{i}}$, otherwise, sample $\bm{v} \sim \mathbf{U}(0,1)^{d_i}$ and $w' \sim \mathbf{U}(0,1)$, set ${\bm{y}_{i}} = \psi_i({\bm{y}_{-i}},\bm{v})$ until $q_i({\bm{y}_{i}})w' > p_i({\bm{y}_{i}})$}}
			\ENDFOR
			\STATE{\black{Set ${\bm{X}_{t+1}} = ({\bm{x}_1}, \ldots,{\bm{x}_s}), {\bm{Y}_{t}} = ({\bm{y}_1}, \ldots,{\bm{y}_s})$. If ${\bm{X}_{t+1}} = {\bm{Y}_{t}}$, set $\tau = t$ and $t = t+1$}} 
			\ELSE
			\STATE{Sample $\bm{V}_{t+1} \sim \mathbf{U}(0,1)^d$, set ${\bm{X}_{t+1}} = \phi({\bm{X}_t},\bm{V}_{t+1}), {\bm{Y}_t} = {\bm{X}_{t+1}}$ and $t = t+1$}
			\ENDIF
			\STATE{Return $F_{k:m}(\mathbf{X},\mathbf{Y})$ by \eqref{eq:Fkm}}
		\end{algorithmic}
	\end{algorithm}
	Following \citep{jacob2020}, in each iteration of Algorithm \ref{alg:ubest_rand}, we apply the maximal coupling method to every $d_i$-dimensional block. Since Gibbs sampling relies on full conditional distributions, the coupling of one block occurs after the coupling of other related blocks. Although this approach does not have a theoretical guarantee of producing complete coupling of the chain within a reasonable time, it performs well in practice.
	
	For the chain \(({\bm{X}_t})_{t\ge1}\), its variable matrix is
	
	\small
	\[V_\mathbf{X} = 
	\left[ {\begin{array}{*{1}{c}}
			\bm{V}_{1}\\
			\bm{V}_{2}\\
			\vdots \\
			\bm{V}_{\kappa}
	\end{array}} \right] = 
	\left[ {\begin{array}{*{4}{c}}
			v_1&v_2& \cdots &v_d\\
			v_{d + 1}&v_{d + 2}& \cdots &v_{2d}\\
			\vdots & \vdots & \ddots & \vdots \\
			v_{(\kappa - 1)d + 1}&v_{(\kappa - 1)d + 2}& \cdots &v_{\kappa d}
	\end{array}} \right],\]
	\normalsize
	where $\kappa = \max(m,\tau)$, and $\bm{V}_{i} =(v_{(i-1)d+1},\dots,v_{id})\stackrel{iid}\sim \mathbf{U}(0,1)^{d}$ for $1\le i \le \kappa$. In practice, when using array-(W)CUD sequences, it is necessary to determine the length of the sequences in advance. Given $d$ and $m$, the number of columns in $V_\mathbf{X}$ is fixed, but the number of rows in $V_\mathbf{X}$ is related to the random variable $\tau$. Fortunately, since it takes the maximum value of $m$ and $\tau$, we can set the number of rows of the variable matrix driven by an array-(W)CUD sequence to be \(m\). If the selected \(m < \tau\), we use an IID sequence with an unspecified length to fill in the gap. We refer to this method as ``CUD-IID” method. We will demonstrate the rationality behind this approach below. Nevertheless, in the maximal coupling method, the sampling of the chain $\mathbf{Y}$ is similar to the acceptance-rejection sampling. The number of variables needed for each update of \(({\bm{Y}_t})_{t\ge0}\) is uncertain. Moreover, the upper bound of $E\left[\mathrm{BC}_{k:m}^2\right]$ is $O(N^{-2})$ in unbiased MCMC \cite{jacob2020}, indicating the bias correction term $\mathrm{BC}_{k:m}$ is not the dominant term for the variance of the estimator $F_{k:m}(\mathbf{X},\mathbf{Y})$. If we attempt to change the driving sequence of $\mathbf{Y}$, the gain may be minimal. Hence, we do not modify the sampling process of the chain $\mathbf{Y}$.
	
	We know that when replacing IID sequences in the MCMC sampling process with other sequences, it is preferable to use (W)CUD sequences, theoretically ensuring consistency \citep{chen2011,tribble2008}. In the aforementioned process of selecting the row number in $V_\mathbf{X}$, we utilize the concatenation of different array-(W)CUD sequences. This raises a question: Does the concatenated sequence retain the array-(W)CUD property? If not, this approach is risky. Theorem \ref{thm:cud} answers this question, establishing that the naturally concatenated sequence of two array-(W)CUD sequences remains an array-(W)CUD sequence. To prove Theorem~\ref{thm:cud}, we begin by presenting a useful lemma introduced in
	\citep{chentsov1967,tribble}.
	\begin{lemma}\label{lem:cud}
		The sequence $(v_i)_{i\ge 1}$ is CUD if and only if for arbitrary integers $1\le l\le d$, the \black{non-overlapping} sequence \black{$\{\bar{\bm{z}}_{i}\}$} defined by $\black{\bar{\bm{z}}_{i}} = (v_{id-l+1},v_{id-l+2},\ldots,v_{(i+1)d-l})$ satisfies
		$$
		\lim\limits_{n\to\infty}D_N^*(\black{\bar{\bm{z}}_{1},\bar{\bm{z}}_{2},\ldots,\bar{\bm{z}}_{N}}) = 0,~\text{as}~N\to\infty.
		$$  
		where $D_N^*$ is the star discrepancy defined in \eqref{eq:stardis}. An analogous equivalence holds for WCUD sequences.
	\end{lemma}	 
	
	Lemma \ref{lem:cud} illustrates that (W)CUD sequences exhibit good balance not only for overlapping blocks but also for non-overlapping blocks with any offset. The result of this lemma also holds for array-(W)CUD as shown in \citep{tribble}.
	\begin{theorem}\label{thm:cud}
		Suppose that $(v_{N,1},v_{N,2},\ldots,v_{N,N})$ and $(u_{M,1},u_{M,2},\ldots,u_{M,M})$ are two array-(W)CUD sequences, then $(w_{K,1},w_{K,2},\ldots,w_{K,K})$ is also an array-(W)CUD sequence with $K=N+M$ and $w_{K,i} = v_{N,i},w_{K,N+j} = u_{M,j},\ i=1,2,\ldots,N,\ j=1,2,\ldots,M$.
	\end{theorem}
	\begin{proof}
		We only consider the array-CUD sequence, and the proof of the array-WCUD sequence is similar. By the definition of array-CUD sequences given in \eqref{eq:arrcud}, for any $d \ge 1$, as $N\to \infty$ through the values in an infinite positive integers set $\mathscr{N}$, we have 
		\begin{equation*}
			D_{N-d+1}^*\left((v_{N,1},\ldots,v_{N,d}),(v_{N,2},\ldots,v_{N,d+1}), \ldots, (v_{N,N-d+1},\ldots,v_{N,N})\right)\to 0.
		\end{equation*}
		By Lemma \ref{lem:cud} with $l=d$, it is equivalent to 
		\begin{equation}\label{eq:DN}
			D_{N'}^*\left(\black{\bar{\bm{x}}_{1},\bar{\bm{x}}_{2},\ldots,\bar{\bm{x}}_{{N'}}}\right)\to 0,~\text{as}~N'\to\infty,
		\end{equation}
		where $\black{\bar{\bm{x}}_{i}} = (v_{N,(i-1)d+1},v_{N,(i-1)d+2},\ldots,v_{N,id}),\ i=1,2,\ldots,N'$ and $N' = \lfloor N/d \rfloor$. Similarly, by Lemma \ref{lem:cud} with $l= N \mod d$, we also have
		\begin{equation}\label{eq:DM}
			D_{M'}^*\left(\black{\bar{\bm{y}}_{1},\bar{\bm{y}}_{2},\ldots,\bar{\bm{y}}_{{M'}}}\right)\to 0,~\text{as}~M'\to\infty,
		\end{equation}
		where $\black{\bar{\bm{y}}_{j}} = (u_{M,jd-l+1},u_{M,jd-l+2},\ldots,u_{M,jd-l+d}),\ j=1,2,\ldots,M'$ and $M' = \lfloor (M-d+l)/d \rfloor$.
		Let $\black{\bm{\bar z}} = (v_{N,N'd+1},v_{N,N'd+2},\ldots,v_{N,N},u_{M,1},u_{M,2},\ldots,u_{M,d-l})$, and $K' = M'+N'+1$. By the definition of star discrepancy given in \eqref{eq:stardis} and absolute value inequality, we obtain that
		\begin{equation}\label{eq:dis}
			\begin{aligned}
				&D_{K'}^*\left(\black{\bar{\bm{x}}_{1},\bar{\bm{x}}_{2},\ldots,\bar{\bm{x}}_{{N'}}, \bm{\bar z}, \bar{\bm{y}}_{1},\bar{\bm{y}}_{2},\ldots,\bar{\bm{y}}_{{M'}}}\right)\\
				\le&\frac{N'}{K'}D_{N'}^*\left(\black{\bar{\bm{x}}_{1},\bar{\bm{x}}_{2},\ldots,\bar{\bm{x}}_{{N'}}}\right)+\frac{M'}{K'}D_{M'}^*\left(\black{\bar{\bm{y}}_{1},\bar{\bm{y}}_{2},\ldots,\bar{\bm{y}}_{{M'}}}\right) + \frac{1}{K'}D_1^*(\black{\bm{\bar z}}).\textbf{}
			\end{aligned}
		\end{equation}
		Hence, by \eqref{eq:DN} and \eqref{eq:DM}, it follows that $\text{as}~N'\to \infty,M'\to\infty,$ i.e., $K'\to\infty$, $$D_{K'}^*\left(\black{\bar{\bm{x}}_{1},\bar{\bm{x}}_{2},\ldots,\bar{\bm{x}}_{{N'}}, \bm{\bar z}, \bar{\bm{y}}_{1},\bar{\bm{y}}_{2},\ldots,\bar{\bm{y}}_{{M'}}}\right) \to 0.$$ Then as $K\to\infty$, 
		\begin{align*}
			D_{K-d+1}^*\left((w_{K,1},\ldots,w_{K,d}),(w_{K,2},\ldots,w_{K,d+1}), \ldots, (w_{K,K-d+1},\ldots,w_{K,K})\right)\to 0,
		\end{align*}
		with $K=N+M$ and $w_{K,i} = v_{N,i},w_{K,N+j} = u_{M,j},\ i=1,2,\ldots,N,\ j=1,2,\ldots,M$. Therefore, $(w_{K,1},w_{K,2},\ldots,\ldots,w_{K,K})$ is an array-CUD sequence. 
	\end{proof}
	
	Hence, the sequences generated by CUD-IID method remain array-(W)CUD sequences. Therefore, using the CUD-IID method in unbiased MCQMC still preserves consistency.

\section{Convergence analysis}\label{sec:theory}
In this section, we prove the unbiasedness of the proposed estimator and discuss its variance.
We follow the assumptions in \citep{jacob2020} for verifying unbiasedness and bounding variance of our proposed estimators.
\begin{assumption}\label{ass1}
	As $t \rightarrow \infty,\  \mathbb{E}\left[f\left(\black{{\bm{X}_t}}\right)\right] \rightarrow \mathbb{E}_\pi[f(\black{\bm{X}})]$ for a real-valued function $f$. Furthermore, there are $\eta>0$ and $D<\infty$ such that $\mathbb{E}\left[\left|f\left(\black{{\bm{X}_t}}\right)\right|^{2+\eta}\right] \leq D$ for all $t \geq 0$.
\end{assumption}
\begin{assumption}\label{ass2}
	The two chains \((\black{{\bm{X}_t}})_{t\ge0}\) and \((\black{{\bm{Y}_t}})_{t\ge0}\) are such that there exists a random variable (the first meeting time) $$\tau=\inf \left\{t \geq 1: \black{{\bm{X}_t}}=\black{{\bm{Y}_{t-1}}}\right\}.$$ Moreover, the meeting time satisfies $\mathbb{P}(\tau>t) \leq C \delta^t$ for all $t \geq 0$, some constants $C<\infty$ and $\delta \in(0,1)$.
\end{assumption}
\begin{assumption}\label{ass3}
	The two chains \((\black{{\bm{X}_t}})_{t\ge0}\) and \((\black{{\bm{Y}_t}})_{t\ge0}\) stay together after the meeting time $\tau$, i.e., $\black{{\bm{X}_t}}=\black{{\bm{Y}_{t-1}}}$ for all $t \geq \tau$.
\end{assumption}

\begin{proposition}\label{pro:ub}
	Under Assumptions \ref{ass1}-\ref{ass3} for all $m\ge k\ge0$, the estimator $F_{k:m}(\black{\mathbf{X}},\black{\mathbf{Y}})$ defined in \eqref{eq:Fkm}, after replacing the original IID sequences with the randomized array-(W)CUD sequences, is an unbiased estimator for $E_{\pi}[f(\black{\bm{X}})]$, and has a finite expected computing time as well as a finite variance. 
\end{proposition}

Following \cite{chen, harase2021, harase2024,tribble}, we randomize array-(W)CUD sequences. Let $\bm{u}_i$ be the $i$-th row of the variable matrix. After randomization, the marginal distributions of \(\{\bm{u}_i\}_{i=1}^{m}\) are all \(\black{\mathbf{U}}(0,1)^d\). Therefore, replacing the original IID sequences with the randomized array-(W)CUD sequences for each state \(\black{{\bm{X}_t}}\) does not alter their marginal distribution and properties. Here, we briefly review the proof of Proposition 1 in \citep{jacob2020} which does not rely on the Markov property. By following their proof, we obtain the results of Proposition \ref{pro:ub}. 

\begin{proof}[Proof of Proposition \ref{pro:ub}]
	For convenience, we just consider \(F_0(\black{\mathbf{X}},\black{\mathbf{Y}})\). The results for \(F_k(\black{\mathbf{X}},\black{\mathbf{Y}})\) and \(F_{k:m}(\black{\mathbf{X}},\black{\mathbf{Y}})\) are similar. Define \(\Delta_0=f(\black{\bm{X}}_0), \Delta_t=f(\black{{\bm{X}_t}})-f(\black{{\bm{Y}_{t-1}}}), t\ge1\), and \(F_0^n(\black{\mathbf{X}},\black{\mathbf{Y}}) = \sum_{t=0}^{n}{\Delta_t}.\) By Assumption \ref{ass2}, it follows that \(E[\tau] < \infty\), indicating that the expected computation time for \(F_0(\black{\mathbf{X}},\black{\mathbf{Y}})\) is finite. Combined with Assumption \ref{ass3}, as \(n \to \infty\), \(F_0^n(\black{\mathbf{X}},\black{\mathbf{Y}}) \to F_0(\black{\mathbf{X}},\black{\mathbf{Y}})\) almost surely. Denote \(L_2\) as the complete space of random variables with finite second moments. By proving that \(\{F_0^n(\black{\mathbf{X}},\black{\mathbf{Y}})\}_{n\ge0}\) is actually a Cauchy sequence on \(L_2\), it is shown that \(F_0(\black{\mathbf{X}},\black{\mathbf{Y}})\in L_2\). Therefore, \(F_0(\black{\mathbf{X}},\black{\mathbf{Y}})\) has a finite expected value and variance. Combined with \eqref{eq:ub} and Assumption \ref{ass1}, it is deduced that \(E[F_0(\black{\mathbf{X}},\black{\mathbf{Y}})]=E_\pi[f(\black{\bm{X}})]\).  This completes the proof.
\end{proof}  

Next, we consider the variance of $F_{k:m}(\black{\mathbf{X}},\black{\mathbf{Y}})$ in unbiased MCQMC. Following \citep{jacob2020}, the variance of $F_{k:m}(\black{\mathbf{X}},\black{\mathbf{Y}})$ can be bounded via
\begin{shrinkeq}{-1ex}
	\begin{align}\label{eq:varf}
		V\left[F_{k:m}(\black{\mathbf{X}},\black{\mathbf{Y}})\right]\leq\text{MSE}_{k:m}+2\sqrt{\text{MSE}_{k:m}}\sqrt{E[\text{BC}_{k:m}^2]}+E[\text{BC}_{k:m}^2],
	\end{align}
\end{shrinkeq}
where $\text{MSE}_{k:m} = E[(\text{MCMC}_{k:m}-E_\pi[f(\black{\bm{X}})])^2]$. Jacob et al. \cite{jacob2020} in Proposition 3 introduced a geometric drift condition on the Markov kernel $P$ and constructed martingales by using the Markov property, consequently providing an upper bound for $E[\text{BC}_{k:m}^2]$. \black{Unfortunately, the chain $({\bm{X}_t})_{t\ge0}$ no longer possesses the Markov property in unbiased MCQMC.}
Therefore, we cannot apply the results in \citep{jacob2020} to derive an upper bound for $E[\text{BC}_{k:m}^2]$. 
We next provide a new upper bound for $E[\text{BC}_{k:m}^2]$ under Assumptions \ref{ass1}-\ref{ass3}.
\begin{proposition}\label{pro:BC}
	Under Assumptions \ref{ass1}-\ref{ass3}, for all $k\ge0$, the sample size $N>0$, and $m= N+k-1$, we have
	\begin{shrinkeq}{-1ex}
		\begin{equation}\label{eq:bc}
			E\left[\mathrm{BC}_{k:m}^2\right] \le C_\delta\frac{\hat{\delta}^{2k}}{N^2},
		\end{equation}
	\end{shrinkeq}
	where 
	\begin{shrinkeq}{-0.5ex}
		$$C_{\delta} = 4D^{\frac{2}{2+\eta}}C^{\frac{\eta}{2+\eta}}\frac{3+2\left(\left(-\frac{\hat{\delta}^{(-1/\ln(\hat{\delta}))}}{\ln(\hat{\delta})}\right)\frac{1-\hat{\delta}}{\hat{\delta}}\right)^2}{\left(\frac{(1-\hat{\delta})^2}{\hat{\delta}}\right)^2},$$
	\end{shrinkeq}
	$\hat{\delta} = \delta^{\frac{\eta}{4+2\eta}}$, $D$ and $\eta$ are defined in Assumption \ref{ass1}, and $C$ and $\delta$ are defined in Assumption~\ref{ass2}.
\end{proposition}

To prove Proposition \ref{pro:BC}, we require a useful lemma as follows.
\begin{lemma}\label{lem:bc1}
	Under Assumptions \ref{ass1}-\ref{ass3}, for all $k\ge0$ and sample size $N>0$, let $m= N+k-1$, then we have
	\begin{equation}\label{eq:bc1}
		E\left[\left(\sum_{l=k+1}^{\infty}\frac{l-k}{m-k+1}\left|f(\black{{\bm{X}_l}})-f(\black{{\bm{Y}_{l-1}}})\right|\right)^2\right] \le \tilde{C}\left(\frac{\hat{\delta}^{k+1}}{N(1-\hat{\delta})^2}\right)^2,
	\end{equation}
	and
	\begin{equation}\label{eq:bc2}
		E\left[\left(\sum_{l=k+1}^{\infty}\left|f(\black{{\bm{X}_l}})-f(\black{{\bm{Y}_{l-1}}})\right|\right)^2\right] \le \tilde{C}\left(\frac{\hat{\delta}^{k+1}}{1-\hat{\delta}}\right)^2,
	\end{equation}
	where $\tilde{C} = 4D^{\frac{2}{2+\eta}}C^{\frac{\eta}{2+\eta}}$, $\hat{\delta} = \delta^{\frac{\eta}{4+2\eta}}$ with $D$ and $\eta$ defined in Assumption \ref{ass1} and $C$ and $\delta$ defined in Assumption \ref{ass2}.
\end{lemma}
\begin{proof}
	Denote $\Delta_t = f(\black{{\bm{X}_{t}}})-f(\black{{\bm{Y}_{t-1}}}),t\ge1$. We have
	\begin{align*}
		E\left[\left(\sum_{l=k+1}^{\infty}\frac{l-k}{m-k+1}\left|f(\black{{\bm{X}_l}})-f(\black{{\bm{Y}_{l-1}}})\right|\right)^2\right] &= E\left[\left(\sum_{h=1}^{\infty}\frac{h}{N}\left|\Delta_{h+k}\right|\right)^2\right] \\&=
		\frac{1}{N^2}\sum_{i=1}^{\infty}\sum_{j=1}^{\infty}ijE\left[\left|\Delta_{i+k}\right|\cdot\left|\Delta_{j+k}\right|\right].
	\end{align*}
	By the proof of Proposition 1 in \citep{jacob2020}, we have 
	\begin{align*}
		E\left[\left|\Delta_{i+k}\right|\cdot\left|\Delta_{j+k}\right|\right] \le \left(E\left[\left|\Delta_{i+k}\right|^2\right]\cdot E\left[\left|\Delta_{j+k}\right|^2\right]\right)^{1/2} \le \left(\tilde{C}^2\hat{\delta}^{2(k+i)}\hat{\delta}^{2(k+j)}\right)^{1/2},
	\end{align*}
	where $\tilde{C} = 4D^{\frac{2}{2+\eta}}C^{\frac{\eta}{2+\eta}}$ and $\hat{\delta} = \delta^{\frac{\eta}{4+2\eta}}$.
	Hence, it follows that
	\begin{align*}
		&E\left[\left(\sum_{h=1}^{\infty}\frac{h}{N}\left|\Delta_{h+k}\right|\right)^2\right]
		\le \frac{1}{N^2}\sum_{i=1}^{\infty}\sum_{j=1}^{\infty}ij \left(\tilde{C}^2\hat{\delta}^{2(k+i)}\hat{\delta}^{2(k+j)}\right)^{1/2}
		\le \tilde{C}\left(\frac{\hat{\delta}^{k+1}}{N(1-\hat{\delta})^2}\right)^2.
	\end{align*}
	Similarly,  we can  obtain the result of \eqref{eq:bc2}.
	This completes the proof.
\end{proof}

Now we are ready to prove Proposition \ref{pro:BC}.
\begin{proof}[Proof of Proposition \ref{pro:BC}]
	Denote $\Delta_t = f(\black{{\bm{X}_{t}}})-f(\black{{\bm{Y}_{t-1}}}),t\ge1$. By $F_{k:m}(\black{\mathbf{X}},\black{\mathbf{Y}})$ defined in \eqref{eq:Fkm}, we know that
	\begin{align*}
		\text{BC}_{k:m}^2&= \sum_{l=k+1}^{\tau-1}\min\left(1,\frac{l-k}{m-k+1}\right)\Delta_l
		:= \begin{cases}
			\Rmnum{1}, \text{~if~} \tau \le m+1,\\
			\Rmnum{2}_1+\Rmnum{2}_2, \text{~if~} \tau > m+1,\\
		\end{cases}
	\end{align*}
	where $\Rmnum{1} = \sum_{l=k+1}^{\tau-1}\frac{l-k}{m-k+1}\Delta_l$, $\Rmnum{2}_1 = \sum_{l=k+1}^{m}\frac{l-k}{m-k+1}\Delta_l$ and 
		$\Rmnum{2}_2 = \sum_{l=m+1}^{\tau-1}\Delta_l.$
	Note that 
	\begin{equation}\label{eq:0}
		\begin{aligned}
			E\left[\text{BC}_{k:m}^2\right]
			&= E\left[\text{BC}_{k:m}^21_{\{\tau\le m+1\}}\right]
			+ E\left[\text{BC}_{k:m}^21_{\{\tau > m+1\}}\right]\\
			&= E\left[\Rmnum{1}^21_{\{\tau\le m+1\}}\right]
			+ E\left[\left(\Rmnum{2}_1+\Rmnum{2}_2\right)^21_{\{\tau > m+1\}}\right]\\
			&\le E\left[\Rmnum{1}^2\right]+E\left[\left(\Rmnum{2}_1+\Rmnum{2}_2\right)^2\right]\\
			& \le E\left[\Rmnum{1}^2\right]+2\left(E\left[\Rmnum{2}_1^2\right]+E\left[\Rmnum{2}_2^2\right]\right),
		\end{aligned}
	\end{equation}
	where $1_{\{\tau\le m+1\}}$ is a \black{indicator} function.
	By Lemma \ref{lem:bc1}, we have
	\begin{shrinkeq}{-0.5ex}
		\begin{align}\label{eq:1}
			E\left[\Rmnum{1}^2\right] \le E\left[\left(\sum_{l=k+1}^{\infty}\frac{l-k}{m-k+1}\left|\Delta_l\right|\right)^2\right] \le  \tilde{C}\left(\frac{\hat{\delta}^{k+1}}{N(1-\hat{\delta})^2}\right)^2.
		\end{align}
	\end{shrinkeq}
	Similarly, we find that
	\begin{shrinkeq}{-0.5ex}
		\begin{align}\label{eq:2}
			E\left[\Rmnum{2}_1^2\right]
			&\le E\left[\left(\sum_{l=k+1}^{\infty}\frac{l-k}{m-k+1}\left|\Delta_l\right|\right)^2\right] \le \tilde{C}\left(\frac{\hat{\delta}^{k+1}}{N(1-\hat{\delta})^2}\right)^2,
		\end{align}
	\end{shrinkeq}
	and
	\begin{shrinkeq}{-0.5ex}
		\begin{align}\label{eq:3}
			E\left[\Rmnum{2}_2^2\right] 
			&\le E\left[\left(\sum_{l=m+1}^{\infty}\left|\Delta_l\right|\right)^2\right] \le  \tilde{C}\left(\frac{\hat{\delta}^{m+1}}{1-\hat{\delta}}\right)^2 = \left(N\hat{\delta}^{N}\frac{1-\hat{\delta}}{\hat{\delta}}\right)^2 \tilde{C}\left(\frac{\hat{\delta}^{k+1}}{N(1-\hat{\delta})^2}\right)^2.
		\end{align}
	\end{shrinkeq}
	Therefore, by \eqref{eq:0}--\eqref{eq:3}, we have
	\begin{shrinkeq}{-0.5ex}
		\begin{align*}
			E\left[\text{BC}_{k:m}^2\right]
			\le \left(3+2\left(N\hat{\delta}^{N}\frac{1-\hat{\delta}}{\hat{\delta}}\right)^2\right)\tilde{C}\left(\frac{\hat{\delta}^{k+1}}{N(1-\hat{\delta})^2}\right)^2 
			\le C_\delta\frac{\hat{\delta}^{2k}}{\left(N^2\right)}.
		\end{align*}
	\end{shrinkeq}
	In the last inequality, we have used the fact that the function $x\hat{\delta}^{x}$ reaches its maximum at $-1/\ln(\hat{\delta})$.	
	This completes the proof.
\end{proof}

	The result of Proposition \ref{pro:BC} is similar to that of Proposition 3 in \citep{jacob2020}. But our proof is different from that of \citep{jacob2020}. More importantly, it does not require the Markov property used in \citep{jacob2020}, implying that \eqref{eq:bc} holds for unbiased MCQMC. Hence, the difference in variance between unbiased MCQMC and unbiased MCMC is attributed to the term 
	\[
	\text{MSE}_{k:m} = E\left[\left(\frac{1}{m-k+1}\sum\limits_{i=k}^mf(\black{{\bm{X}_i}})-\pi(f)\right)^2\right] = E\left[\left(\frac{1}{N}\sum\limits_{i=1}^Nf({\bm{X}_i}^{\prime})-\pi(f)\right)^2\right],
	\]
	where ${\bm{X}_i}^{\prime} = \black{{\bm{X}_{i+k-1}}}$. 
	Combined with \black{Algorithm \ref{alg:ubest_rand}}, we know that considering $\text{MSE}_{k:m}$ in unbiased MCMC is equivalent to considering it in MCMC. 
	Similarly, after replacing IID sequences with array-(W)CUD sequences, considering $\text{MSE}_{k:m}$ in unbiased MCQMC is also equivalent to considering it in MCQMC. This is equivalent to seeking the convergence rate within the MCQMC framework. 
	However, in MCQMC, as highlighted in the introduction, the theoretical foundation for the convergence rate of MCQMC is presently limited and remains an open and challenging research problem. The works \cite{dick2014,dick2016} 
	showed that there exist driving sequences such that MCQMC yields an error rate of $O(N^{-1+\epsilon})$.
	If MCQMC achieves an RMSE rate of $O(N^{-1+\epsilon})$, it then follows immediately from Proposition \ref{pro:BC} and \eqref{eq:varf} that this rate is retained for unbiased MCQMC. However, the driving sequences in \cite{dick2014,dick2016} are not explicitly constructed. Studying the convergence rate of MCQMC with explicitly constructed driving sequences would contribute to addressing the convergence rate theory for unbiased MCQMC. This is left for future research. On the other hand, \citep{harase2021,harase2024} demonstrated a nearly \(O(N^{-1})\) convergence rate in their numerical examples for commonly used CUD sequences,  without a theoretical guarantee. It is thus expected that MCQMC converges faster than MCMC in the unbiased scheme, as shown by the numerical results in Section \ref{sec:experiments}.

\section{Numerical experiments}\label{sec:experiments}
In this section, we study three types of Bayesian regression problems to demonstrate the performance of unbiased MCQMC across different datasets and dimensions. Let $\bm{D} \in \mathbb{R}^{n\times p}$ be the observation matrix of predictor variables, $\bm{D}_i\in \mathbb{R}^{1\times p}$ be the $i$-th row of $\bm{D}$ and $\bm{\beta} =(\beta_1,\beta_2,\ldots,\beta_p)^{T}\in\mathbb{R}^{p\times 1}$ be the regression parameters. We use Gibbs samplers to estimate the posterior mean of $\bm{\beta}$.

To estimate the variances of the estimators, we use the randomized array-(W)CUD sequences. Following \citep{harase2021}, we adopt the digital shifts method \citep{owenqmc} on Harase's method. Additionally, following \citep{tribble2008}, we apply the random shifts method \citep{cranley1976} to Liao's method. In Liao's method, we use the classical Sobol' sequence with direction numbers from \citep{joe2003}, where the dimension equals the number of array-WCUD points used in the update function, and the number of Sobol' points is $N$. For $N \in \{2^{10},2^{11},\ldots,2^{32}\}$, we use both Harase's method and Liao's method to compare their performance in unbiased MCQMC. For other sample sizes, due to the constraints on the sample size of Harase's method, we only use Liao's method. 

For a burn-in period $k$ and a sample size $N$, we repeat the chains $R$ times independently, resulting in unbiased estimators of the posterior mean of the parameter $\bm\beta$, denoted by $F_{k:m}^{(1)}, \ldots, F_{k:m}^{(R)}$, where $m = N+k-1$. A combined unbiased estimator \(\hat{\mu}_{\text {pool}} = \frac{1}{R} \sum_{r=1}^{R} F_{k:m}^{(r)}\) then serves as the final estimator whose $j$-th component is denoted by $\hat{\mu}_{\text {pool}}^{(j)}$. Thanks to the unbiasedness of the estimators, the component-wise MSE for each $\beta_j$ and the total MSE of the vectorized estimators $\hat{\mu}_{\text {pool}}$ are estimated by
\begin{equation}\label{eq:var}
	\hat{\sigma}_{j}^2= \frac{1}{R(R-1)} \sum_{r=1}^R (F_{k:m}^{(r,j)} - \hat{\mu}_{\text{pool}}^{(j)})^2,~\text{and}~~ \hat{\sigma}_{\text{tot}}^2 = \sum_{j=1}^p \hat{\sigma}_{j}^2, 
\end{equation}
respectively, where $F_{k:m}^{(r,j)}$ denotes the $j$-th component of $F_{k:m}^{(r)}$. We use $\hat{\sigma}_j$ and $\hat{\sigma}_{\text{tot}}$ as the associated empirical RMSEs. Hereafter, we abbreviate unbiased MCMC as ``ubMCMC", unbiased MCQMC as ``ubMCQMC", unbiased MCQMC driven by Liao's method as ``ubMCQMC-L" and unbiased MCQMC driven by Harase's method as ``ubMCQMC-H".

In ubMCQMC, deciding when to use array-WCUD sequences is crucial. In the sampling process of $(\black{{\bm{X}_t}})_{t\ge0}$, directly replacing IID sequences with array-WCUD sequences is inappropriate.  The works \cite{chen2011,harase2021,harase2024,liu2023} recommended to use array-WCUD sequences on MCQMC after a burn-in period because doing otherwise would compromise the overall uniformity of array-WCUD sequences. Motivated by these works, we thus use array-WCUD sequence after the $(k-1)$-th step and use an IID sequence before the $(k-1)$-th step (i.e., the burn-in period), as the samples generated in the previous $k$ steps are actually skipped.

Regarding how to choose an appropriate $k$, \cite{jacob2020} suggested to take a large quantile of $\tau$ in ubMCMC. They conducted a small number (e.g., 1000) of independent experiments as a pilot run to obtain samples of $\tau$, and then took $k$ as the 99\% quantile of these samples. However, such a strategy may not be effective for ubMCQMC. 
It follows from Proposition 3 in \citep{jacob2020} that the bias correction term $\text{BC}_{k:m}$  in \eqref{eq:Fkm} admits $\sqrt{E\left[\text{BC}_{k:m}^2\right]} = O(N^{-1})$, and the RMSE of ubMCMC is thus dominated by the term $\text{MCMC}_{k:m}$, which is of order $O(N^{-1/2})$. Although the bias correction term still contributes $O(N^{-1})$ to the RMSE of ubMCQMC as confirmed by Proposition~\ref{pro:BC}, the term $\text{MCMC}_{k:m}$ often converges faster than that of ubMCMC, with a nearly $O(N^{-1})$ rate for some cases \citep{harase2021,harase2024}. This implies that the bias correction term may affect more significantly the performance of ubMCQMC compared to ubMCMC. 
On the other hand, to capture the tail distribution of $\tau$ well, it often requires a large amount samples of $\tau$ in the pilot run. To alleviate the impact of the bias correction term by a small set of samples of $\tau$, we propose to take a somewhat conservative \( k \) as twice the 99\% quantile of 1000 samples of $\tau$ in the following experiments. For fair comparisons, we still use this choice of $k$ for ubMCMC.
Moreover, we find that the empirical distributions of $\tau$ obtained by IID sequences and array-WCUD sequences are very similar in our experiments. This suggests that one may compute the 99\% sample quantile of $\tau$ based on IID driving sequences.  In Appendix \ref{app:k}, we provide some intuitive experiments on examining the impact of the burn-in period $k$.

\subsection{Bayesian linear regression model}\label{subsec:linear}  The first example is a Bayesian linear regression model which was used in \citep{harase2024}. Given the observations of response variable $\bm{y} = (y_1,y_2,\ldots,y_n)^{T} \in \mathbb{R}^{n\times 1}$, the linear model is expressed as \(y_i= \bm{D}_i\bm{\beta}+\epsilon_i\) where $\epsilon_i \stackrel{iid}\sim \mathcal{N}(0, \sigma^2)$. Assume that $\bm{\beta}$ and $\sigma^2$ have independent prior distributions: 
$$
\bm{\beta} \sim \mathcal{N}(\bm{b}_{0}, \bm{B}_{0}), \quad \sigma^2 \sim \text{IGa}(n_0/2,s_0/2), 
$$
where $\bm{b}_{0} = \bm{0}^{p \times 1}$, $\bm{B}_{0} = \text{diag}(100,100,\ldots,100)$ is a $p\times p$ diagonal matrix with each diagonal element equal to $100$, $n_0 = 5$, $s_0  = 0.01$, and $\text{IGa}(\alpha_1,\alpha_2)$ denotes the inverse Gamma distribution with a shape parameter $\alpha_1$ and a scale parameter $\alpha_2$. The full conditional distributions are as follows,
$$
\bm{\beta}|\sigma^2 \sim \mathcal{N}\left(\bm{b}_1, \bm{B}_1\right), \quad \sigma^2 | \bm{\beta} \sim \text{IGa}\left(n_1/2,s_1/2\right),
$$
where $\bm{b}_1=\bm{B}_1\left(\bm{B}_0^{-1} \bm{b}_0+\sigma^{-2}\bm{D}^{\top}\bm{y}\right)$, $ \bm{B}_1^{-1}=\bm{B}_0^{-1}+\sigma^{-2} \bm{D}^{\top}\bm{D}$, $n_1=n_0+n$ and $s_1=s_0+(\bm{y}-\bm{D} \bm{\beta})^{\top}(\bm{y}-\bm{D} \bm{\beta})$.  To generate a multivariate normal $\bm Q\sim \mathcal{N}(\bm{\mu}$, $\bm{\Sigma})$ in $d$ dimensions, it suffices to take $\bm{Q} = \bm{\mu} + \bm{C}\Phi^{-1}(\bm V)$, where $\bm{C}\bm{C}^{T} = \bm{\Sigma}$, $\bm V \sim \black{\mathbf{U}}(0,1)^d$ and $\Phi^{-1}(\cdot)$ is the inverse CDF function of the standard normal distribution \cite{owenmc}. Hence, the update function for this model is clear. 

We examine two datasets in this model. The first one is the Boston housing dataset analyzed in \citep{harrison1978} with $n=506$ observations on $14$ predictor variables, implying $p = 14$ and the dimension $d=p+1=15$. The second one is the California housing dataset analyzed in \citep{pace1997} with $n=20640$ observations on $9$ predictor variables, implying $p = 9$ and $d=10$. In our experiments, we choose the burn-in period $k = 8$ for the Boston dataset and $k = 6$ for the California dataset based on the aforementioned selecting rule.

We are interested in estimating the posterior mean of the  regression parameter $\bm \beta\in \mathbb{R}^{p\times 1}$. The total RMSEs of the estimators are shown in Figure \ref{Lineara}-\ref{Linearb} for $N \in  \{2^{10},2^{11},\ldots,2^{16}\}$, and the RMSEs for each components of $\bm\beta$ are presented in Figure \ref{Linearc}-\ref{Lineard} with \(N \in \{2^{10},2^{13},2^{16}\}\).
All RMSEs are computed based on $R = 100$ repetitions.
\begin{figure}[htbp]
	\vspace{-4mm}
	\centering
	\subfigure[Total RMSEs for Boston.]{\label{Lineara}
		\begin{minipage}[t]{0.4\linewidth}
			\centering
			\includegraphics[width=0.9\linewidth]{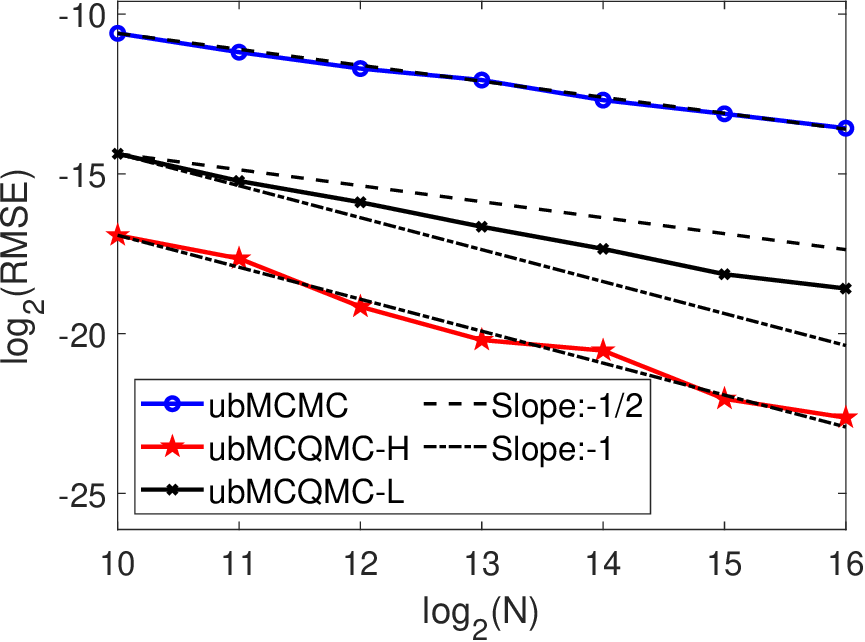}
		\end{minipage}
	}
	\vspace{-3mm}
	\subfigure[Total RMSEs for California.]{\label{Linearb}
		\begin{minipage}[t]{0.4\linewidth}
			\centering
			\includegraphics[width=0.9\linewidth]{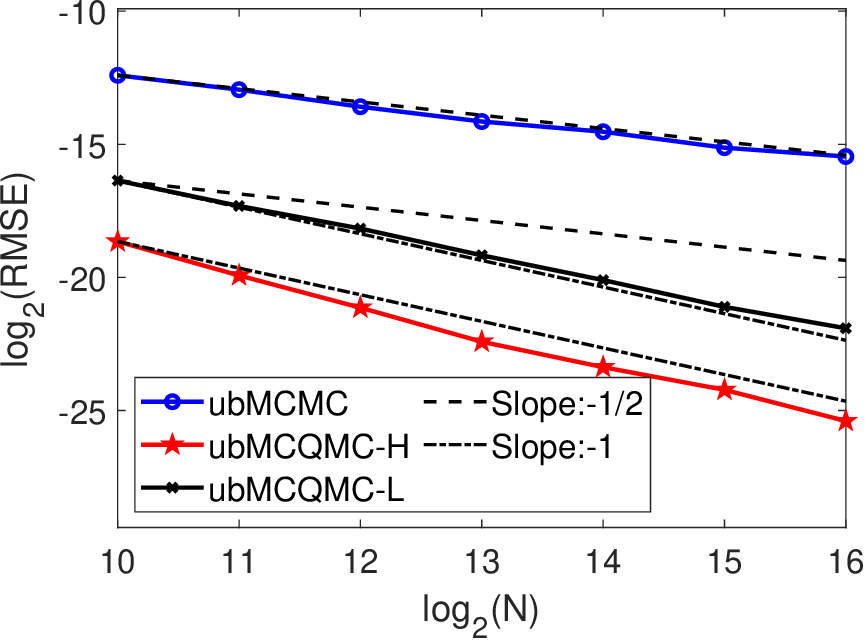}
		\end{minipage}
	}\\
	\subfigure[Component-wise RMSEs for Boston.]{\label{Linearc}
		\begin{minipage}[t]{0.4\linewidth}
			\centering
			\includegraphics[width=0.9\linewidth]{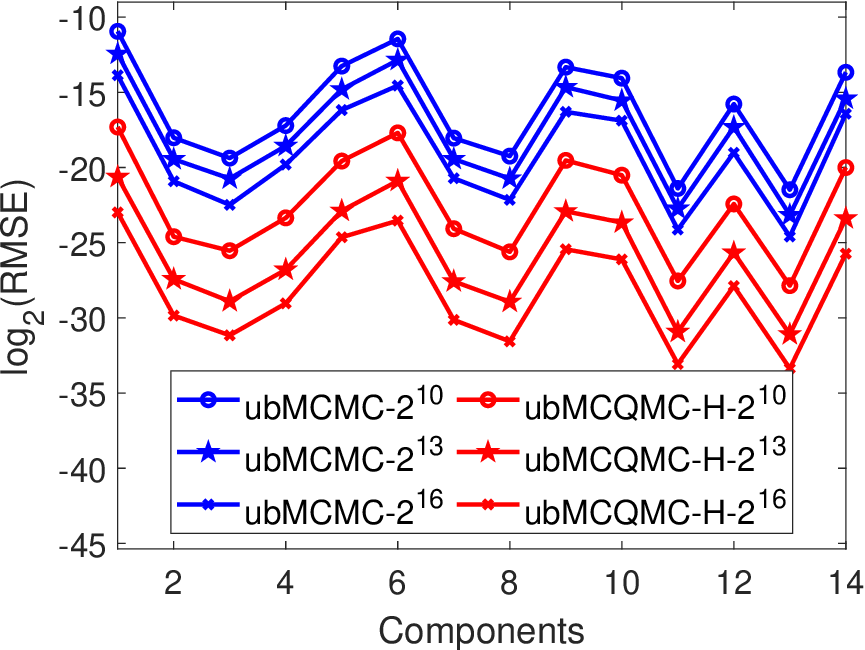}
		\end{minipage}
	}
	\vspace{-3mm}
	\subfigure[Component-wise RMSEs for California.]{\label{Lineard}
		\begin{minipage}[t]{0.4\linewidth}
			\centering
			\includegraphics[width=0.9\linewidth]{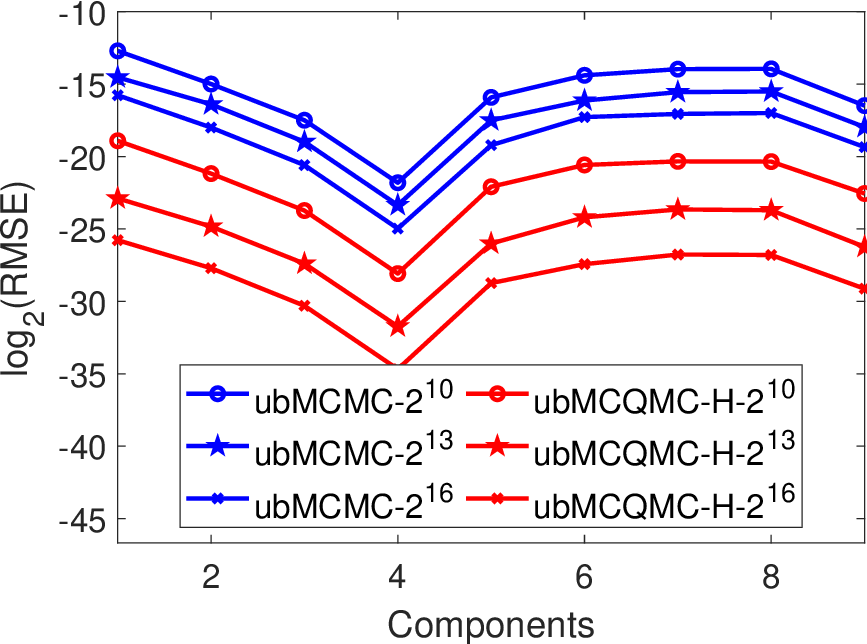}
		\end{minipage}
	}
	\caption{RMSEs for the Boston and California datasets.}
	\label{fig:Linear}
	\vspace{-6mm}
\end{figure}

From Figure \ref{Lineara}-\ref{Linearb}, we observe that the RMSE convergence rate of ubMCQMC-H reaches approximately \(O(N^{-1})\) for both datasets. Based on a simple least square estimation, the rate of ubMCQMC-L attains approximately \(O(N^{-0.71})\) for the Boston dataset and \(O(N^{-1})\) for the California dataset. All of them beat the \(O(N^{-1/2})\) rate of ubMCMC, and ubMCQMC-H performs the best. Figure \ref{Linearc}-\ref{Lineard} further show that ubMCQMC-H yields smaller RMSEs than ubMCMC for all components of $\bm \beta$.


\begin{table}[htbp]
\centering
\caption{Total RMSE reduction factors and runtimes for the Boston and California datasets.}
\scalebox{0.85}{
	\begin{tabular}{ccccccccccccc}
		\hline
		\multirow{3}[6]{*}{$N$} &       & \multicolumn{5}{c}{Boston}    &    & \multicolumn{5}{c}{California} \bigstrut\\
		\cline{3-7} \cline{9-13}   &     & \multicolumn{2}{c}{ubMCMC} &    & \multicolumn{2}{c}{ubMCQMC-H} &    & \multicolumn{2}{c}{ubMCMC} &   & \multicolumn{2}{c}{ubMCQMC-H} \bigstrut\\
		\cline{3-4}\cline{6-7}\cline{9-10}\cline{12-13}          &       & RMSE  & Time (s)  &       & RRF  & Time (s)  &       & RMSE  & Time (s)  &       & RRF  & Time(s) \bigstrut\\
		\cline{1-13}    
		$2^{10}$     &       & 6.43E-04 & 0.01 &       & 79.89 & 0.01 &       & 1.84E-04 & 0.10 &       & 75.75 & 0.10 \bigstrut[t]\\
		$2^{13}$     &       & 2.34E-04 & 0.09 &       & 281.19 & 0.10 &       & 4.59E-05 & 0.80 &       & 309.22 & 0.81 \\
		$2^{16}$     &       & 8.21E-05 & 0.69 &       &  532.60 & 0.81 &       & 4.46E-05 & 6.27 &       & 979.79 & 6.39 \bigstrut[b]\\
		\hline
	\end{tabular}%
}
\label{tab:linear}%
\end{table}%

To better illustrate the improvement of ubMCQMC-H compared to ubMCMC, we provide the total RMSE reduction factors (RRFs)  and runtimes (in seconds) for different values of \(N\) in Table~\ref{tab:linear}. The runtime is averaged over the 100 repetitions. We can see that ubMCQMC significantly accelerates the convergence rate compared to ubMCMC, along with a comparable  cost of ubMCMC. Moreover, ubMCQMC performs better on the California dataset with 20640 observations than on the Boston dataset with 506 observations. In previous MCQMC experiments \citep{harase2021,tribble2008}, the improvement for Gaussian distribution models was significant. Since the posterior distribution of large datasets approximates a Gaussian distribution, this may be amenable to unbiased MCQMC for achieving a large RMSE reduction as observed for the California dataset.

This model for the Boston dataset was studied in \cite{harase2024}, where MCQMC driven by Harase's method with \( k = 5000 \) was used to estimate the posterior mean. We compare the total RMSEs of the standard MCMC and MCQMC which are biased. Here, we use the estimated value of ubMCQMC-H with \( N = 2^{16} \) as the true value to calculate the RMSEs of the biased estimators. It can be observed that the results of MCQMC and unbiased MCQMC are very similar. However, it is noteworthy that the unbiased MCQMC method sets \( k \) just as $8$ rather than $5000$ used in \cite{harase2024}. Moreover, even after discarding the first $5000$ samples in MCQMC, the resulting  estimator is still biased. Unbiased MCQMC is able to address the unbiasedness issue by selecting a very small burn-in period \( k \) for this model.

\begin{figure}[htbp]
\vspace{-2mm}
\centering
\includegraphics[width=0.4\linewidth]{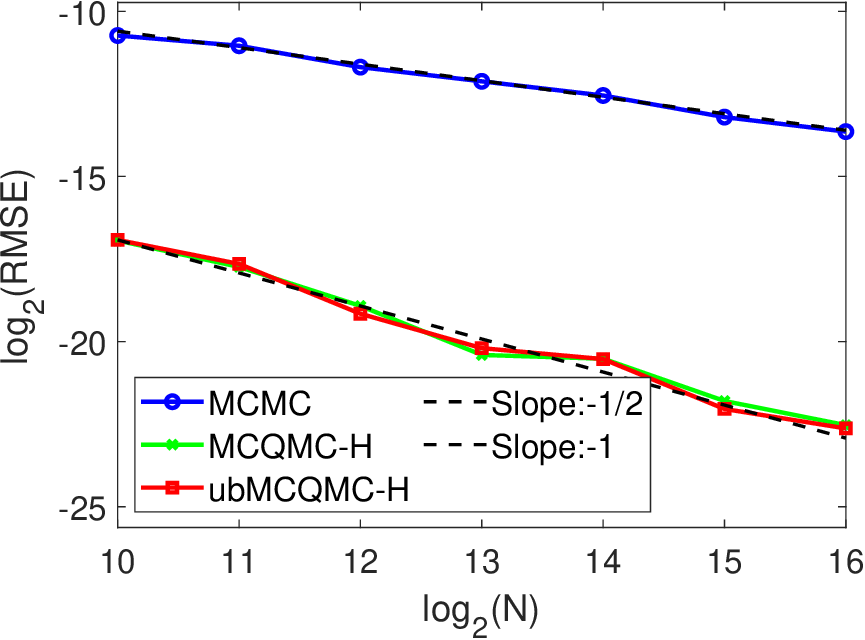}
\caption{Total RMSEs of MCMC, MCQMC-H and ubMCQMC-H for the Boston dataset.}
\label{fig:linearmc}
\vspace{-4mm}
\end{figure}

Note that the introduction of ubMCMC was aimed at achieving parallelization, i.e., maintaining a small $N$ and running $R$ short chains in parallel instead of long chains as in MCMC. Utilizing the property that estimators on each chain are unbiased and the linearity of $R$ in the computational budget, ubMCMC achieves high precision in a relatively short time by increasing $R$. When parallel resources are limited, for ubMCMC, the benefits of running long chains or repeating multiple chains are similar. However, it is obvious that benefiting from the higher convergence order of ubMCQMC, running longer chains can achieve higher precision within a similar budget. 

When parallel resources are abundant, ubMCQMC is able to parallelize shorter chains than ubMCMC to achieve comparable efficiency of MCMC estimators. Following \citep[Section 3.1]{jacob2020}, to compare the efficiency of  the unbiased estimator $F_{k:m}(\black{\mathbf{X}},\black{\mathbf{Y}})$ and the biased estimator $\text{MCMC}_{k:m}$, we define the asymptotic inefficiency of $F_{k:m}(\black{\mathbf{X}},\black{\mathbf{Y}})$ as the product of its variance and its expected cost. Following \citep{jacob2020}, the total cost of $F_{k:m}(\bm{X},\bm{Y})$ is $c[2(\tau-1)+\max(1,m+1-\tau)]$, where $c$ is the cost of ubMCMC or ubMCQMC for a single state. 

For $m \geq k$, when $k$ increases, we expect $(m-k+1)\text{MSE}_{k:m}$ to converge to 
$$V_{k:m} := V\left[(m+k-1)^{-1/2}\sum\limits_{t=k}^m f(\black{{\bm{X}_t}})\right]$$ 
with $\black{{\bm{X}_t}} \sim \pi$. The limit of $V_{k:m}$ as $m \to \infty$ denoted by $V_{\infty}$ is the asymptotic variance of the MCMC estimator. We assume that $(m-k+1)\text{MSE}_{k:m} \approx V_{\infty}$ for $k$ large enough. The loss of efficiency of unbiased methods compared with standard MCMC is defined by 
\begin{equation}\label{eq:efficiency}
\frac{c}{c_0}\times\frac{E[2(\tau-1)+\max(1,m+1-\tau)]\times V[F_{k:m}(\black{\mathbf{X}},\black{\mathbf{Y}})]}{V_{\infty}},
\end{equation}
where $c_0$ is the cost of standard MCMC in each iteration. It is obvious that $c = c_0$ for ubMCMC. For ubMCQMC, although the generation of array-WCUD sequences takes more time than IID sequences, the overhead can be negligible since the main cost of simulation arises from the calculation of the update function in each iteration as observed in Table~\ref{tab:linear}. We thus have $c\approx c_0$ for ubMCQMC if we use the same update function as in ubMCMC. For simplicity, we take $c= c_0$ in calculating loss of efficiency for the Bayesian linear regression model. The expected cost is measured by the number of iterations. And, we should note that the cost of determining $k$ in the pilot run is not counted here.

As discussed in \citep[Section 3.1]{jacob2020}, when $k$ and $m$ are sufficiently large, the loss of efficiency for ubMCMC is close to $1$, implying that the variance brought by the unbiasedness can be eliminated by choosing sufficiently large $k$ and $m$. However, in the setting of parallelization, one might prefer to keep $k$ and $m$ relatively small to achieve a suboptimal efficiency, but generate more independent estimators within a given computing time.

For the Bayesian linear regression model, we choose $N$ to achieve the loss of efficiency near $1$. 
We estimate the expected cost and the total variance  of the estimators based on $1000$ repetitions of the simulation. Meanwhile, the asymptotic total variance $V_{\infty}$ of standard MCMC is obtained from $1000$ long chains with $5\times10^5$ iterations and a burn-in period of $10^3$. 
The expected cost, the total variance of estimators, and the loss of efficiency of ubMCMC and ubMCQMC are given in Table \ref{tab:ineffLinear}, where ``Loeff'' denotes the loss of efficiency and ``SF'' represents the saving factor of sample size $N$ for ubMCQMC compared to ubMCMC. We see that for the Boston dataset, when $N = 120$, the efficiency of ubMCMC is comparable to that of standard MCMC. However, for ubMCQMC, to obtain comparable efficiency of standard MCMC, we just need $N = 6$, resulting in a saving of sample size by a factor of $120/6 = 20$. For the California dataset, we reduce $35$-fold sample size. It is worth mentioning that the sample saving factor and the RMSE reduction factor are two different metrics. The sample size saving factor examines the inefficiency (the product of variance and expected cost) of ubMCQMC in short chains driven by Liao's method, while the RMSE reduction factor evaluates the RMSE of ubMCQMC in long chains driven by Harase's method.
\begin{table}[htbp]
\vspace{-4mm}
\centering
\caption{The loss of efficiency for the Boston and California datasets.}
\scalebox{0.85}{
	\begin{tabular}{cccccccccccccc}
		\hline
		\multirow{2}[4]{*}{} & \multicolumn{6}{c}{Boston}                    &       & \multicolumn{6}{c}{California} \bigstrut\\
		\cline{2-7}  \cline{9-14}      &   $k$  & $N$   & Cost  & Variance   & Loeff & SF &  &   $k$  & $N$   & Cost  & Variance  & Loeff & SF  \bigstrut\\
		\hline
		ubMCMC & 8     & 120    & 129 & 3.32E-07 & 1.07  & \multirow{2}[2]{*}{20} &       & 6 & 210    & 217    & 1.35E-08 & 1.06  & \multirow{2}[2]{*}{35} \bigstrut[t]\\
		ubMCQMC-L &  	8     & 6    & 15 & 2.57E-06 & 0.96 &       &       &    6   & 6     & 13 & 2.00E-07 & 0.94  &  \bigstrut[b]\\
		\hline
	\end{tabular}%
}
\label{tab:ineffLinear}%
\vspace{-4mm}
\end{table}%

\subsection{Bayesian probit regression model} 
Next we consider a Bayesian probit regression model that was used in \citep{albert1993,tribble2008}. We follow the model setting in \citep{tribble2008}. Given the observations of binary response variable $\bm{y} = (y_1,y_2,\ldots,y_n) \in \{0,1\}^{n\times 1}$, the probit model is expressed as \(y_i=1_{\{z_i>0\}}\) with \(z_i \sim \mathcal{N}(\bm{D}_i\bm{\beta},1)\). Our goal is to estimate the posterior mean of \(\bm{\beta}\). Taking a non-informative prior for \(\bm{\beta}\), we have 
\begin{shrinkeq}{-1.5ex}
$$
\begin{aligned}\bm{\beta}|z_1,z_2,\ldots,z_{n} &\sim \mathcal{N}\left(\left(\bm{D}^T\bm{D}\right)^{-1}\bm{D}^T\bm{Z},\left(\bm{D}^T\bm{D}\right)^{-1}\right),\\z_i|\bm{\beta},z_{-i} &\sim \begin{cases} \mathcal{N}\left(\bm{D}_i\bm{\beta},1\right) \text{~\black{truncated} to~} [0,\infty), \text{~if~} y_i = 1, \\\mathcal{N}\left(\bm{D}_i\bm{\beta},1\right) \text{~\black{truncated} to~} (-\infty,0], \text{~if~} y_i = 0, \end{cases}  1\le i \le n,\end{aligned}
$$
\end{shrinkeq}
where $\bm{Z} = (z_1,z_2,\ldots,z_{n})^{T}$ denotes the latent variables. For this model, the dimension of the underlying states is $d=p+n$.
To sample a random variable $Q$ following the truncated distribution of CDF $F(x)$ over the interval $(a,b)$, we  take $Q = F^{-1}(F(a)+(F(b)-F(a))u)$, where $u\sim \black{\mathbf{U}}(0,1)$ \cite{owenmc}. The specific form of the update function is also clear. 

For this model, we consider two datasets with different sizes. The first one is the Vaso constriction dataset from \citep{finney} with $n = 39$ and $p = 3$, resulting in $d = 42$. The second one is the Mroz’s female labor force participation dataset from \citep{mroz1987} and the fitting model for this dataset comes from \citep[Example 17.1]{woold2016} with $n = 753$ and $p = 8$, resulting in $d = 761$. 

We choose $k=82$ for the Vaso dataset and $k = 50$ for the Mroz dataset. 
\begin{figure}[H]
\vspace{-4mm}
\centering
\subfigure[Total RMSEs for Vaso.]{\label{probita}
	\begin{minipage}[t]{0.4\linewidth}
		\centering
		\includegraphics[width=0.9\linewidth]{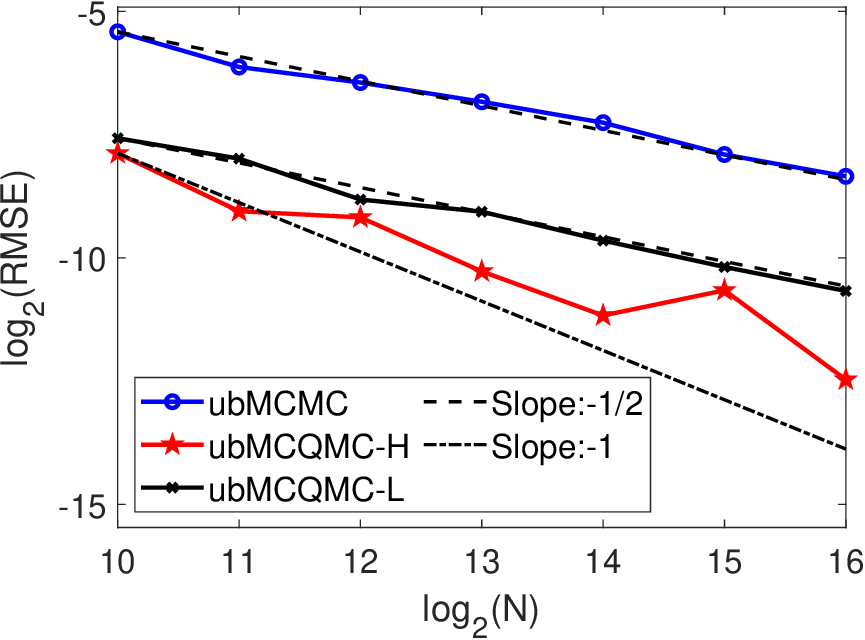}
	\end{minipage}
}
\vspace{-3mm}
\subfigure[Total RMSEs for Mroz.]{\label{probitb}
	\begin{minipage}[t]{0.4\linewidth}
		\centering
		\includegraphics[width=0.9\linewidth]{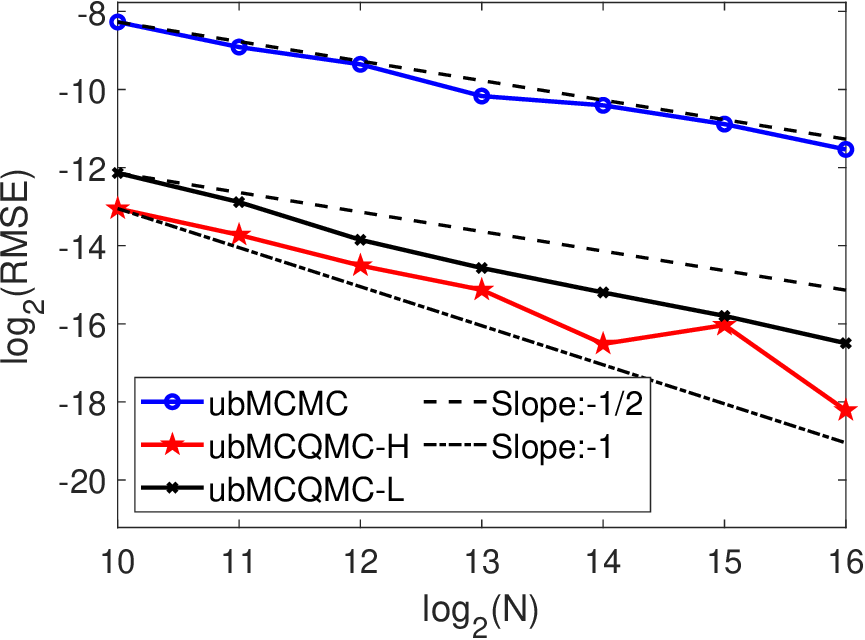}
	\end{minipage}
}\\
\subfigure[Component-wise RMSEs for Vaso.]{\label{probitc}
	\begin{minipage}[t]{0.4\linewidth}
		\centering
		\includegraphics[width=0.9\linewidth]{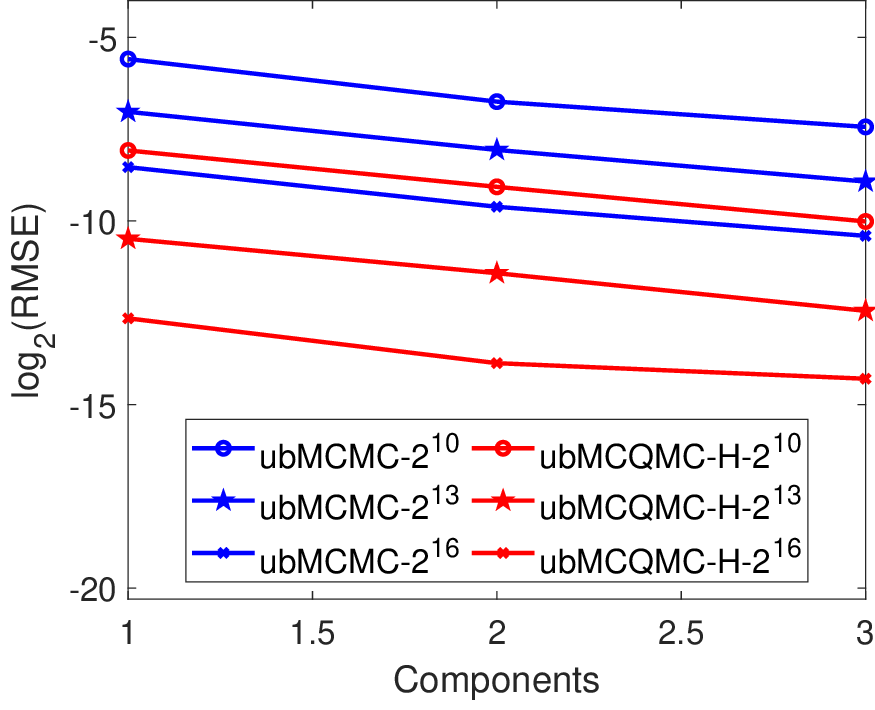}
	\end{minipage}
}
\vspace{-3mm}
\subfigure[Component-wise RMSEs for Mroz.]{\label{probitd}
	\begin{minipage}[t]{0.4\linewidth}
		\centering
		\includegraphics[width=0.9\linewidth]{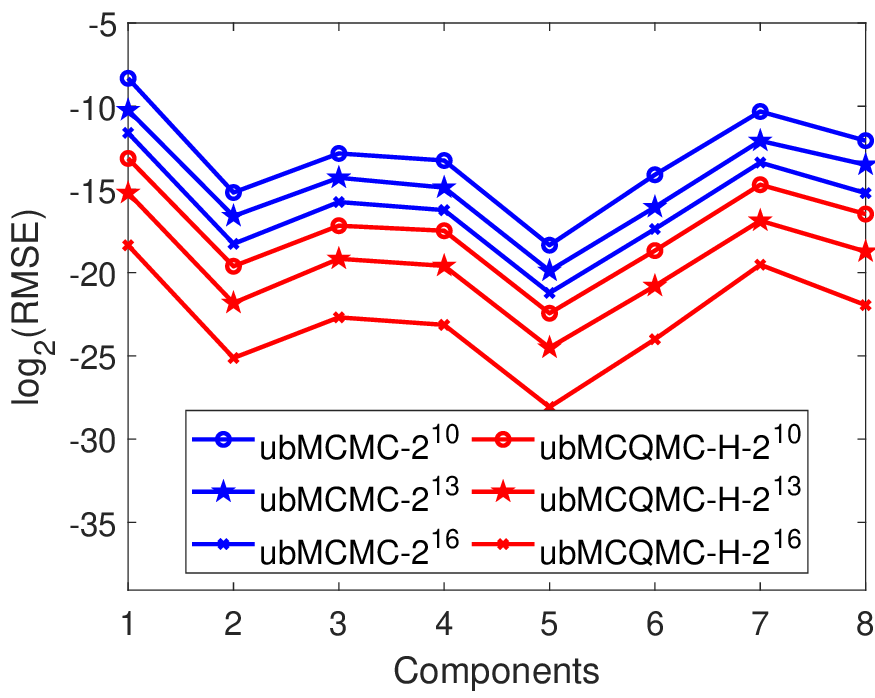}
	\end{minipage}
}
\caption{RMSEs for the Vaso and Mroz datasets.}
\label{fig:probit_tvar}
\end{figure}

Under the same setting of Figure~\ref{fig:Linear}, Figure \ref{probita}-\ref{probitb} show the total RMSEs of ubMCMC and ubMCQMC, and Figure \ref{probitc}-\ref{probitd} show the RMSEs for each component of $\bm\beta$. By a simple least square estimation, we find that the estimated convergence rate of ubMCQMC-H  is nearly $O(N^{-0.67})$ for the Vaso dataset and  $O(N^{-0.79})$ for the Mroz dataset, surpassing the \(O(N^{-1/2})\) rate of ubMCMC. Although ubMCQMC-L yields a worse rate than  ubMCQMC-H, it still performs better than ubMCMC for all cases. From Figure \ref{probitc}-\ref{probitd}, ubMCQMC-H also yields smaller RMSEs than ubMCMC for estimating all components of $\bm\beta$.

\begin{table}[H]
\centering
\caption{Total RMSE reduction factors and runtimes for the Vaso and Mroz datasets.}
\scalebox{0.85}{
	\begin{tabular}{ccccccccccccc}
		\hline
		\multirow{3}[6]{*}{$N$} &       & \multicolumn{5}{c}{Vaso}    &    & \multicolumn{5}{c}{Mroz} \bigstrut\\
		\cline{3-7} \cline{9-13}   &     & \multicolumn{2}{c}{ubMCMC} &    & \multicolumn{2}{c}{ubMCQMC-H} &    & \multicolumn{2}{c}{ubMCMC} &   & \multicolumn{2}{c}{ubMCQMC-H} \bigstrut\\
		\cline{3-4}\cline{6-7}\cline{9-10}\cline{12-13}          &       & RMSE  & Time (s)  &       & RRF  & Time (s)  &       & RMSE  & Time (s)  &       & RRF  & Time (s) \bigstrut\\
		\cline{1-13}    
		$2^{10}$    &     & 2.34E-02 & 0.01 &     & 5.52 & 0.02 &     & 3.24E-03 & 0.16 &    & 27.43 & 0.20 \bigstrut[t]\\
		$2^{13}$    &     & 8.77E-03 & 0.10 &     & 10.86 & 0.13 &    & 8.69E-04 & 1.12 &    & 31.20 & 1.45 \\
		$2^{16}$    &     & 3.07E-03 & 0.81 &     & 17.44 & 1.02 &    & 3.37E-04 & 10.3 &    & 102.66 & 13.6 \bigstrut[b]\\
		\hline
	\end{tabular}%
}
\label{tab:probit_rf}%
\end{table}%

We also provide the total RMSE reduction factors and runtimes (in seconds) for different values of \(N\) in Table \ref{tab:probit_rf}. Compared to Table~\ref{tab:linear}, ubMCQMC gains smaller reduction factors. This may be due to the fact that the dimension $d$ for the Bayesian probit regression model is much larger (up to $761$ for the Mroz dataset), and QMC methods are known to suffer from the curse of dimensionality. On the other hand, since the number of the observations is only 39 for the Vaso dataset, the associated posterior may be far from Gaussian, leading to a worse error rate compared to the Mroz dataset with 753 observations.

Next, we compare the efficiency of ubMCMC and ubMCQMC compared to standard MCMC in Table \ref{tab:ineffprobit} following the same setting of the Bayesian linear regression model. We observe that for the Vaso dataset, when $N = 820$, the efficiency of ubMCMC is comparable to that of standard MCMC. However, for ubMCQMC-L, we just need $N =32$ to achieve comparable efficiency, saving a sample size of $820/32 = 25.6$ times. For the Mroz dataset, ubMCQMC gains a saving factor of $37.5$. 

\begin{table}[htbp]
\centering
\caption{The loss of efficiency for the Vaso and Mroz datasets.}
\scalebox{0.85}{
	\begin{tabular}{cccccccccccccc}
		\hline
		\multirow{2}[4]{*}{} & \multicolumn{6}{c}{Vaso}                    &       & \multicolumn{6}{c}{Mroz} \bigstrut\\
		\cline{2-7}  \cline{9-14}      &   $k$  & $N$   & Cost  & Variance   & Loeff & SF &  &   $k$  & $N$   & Cost  & Variance  & Loeff & SF  \bigstrut\\
		\hline
		ubMCMC & 82    & 820   & 908 & 6.98E-05 & 1.05  & \multirow{2}[2]{*}{25.6} &  & 50    & 600 &  657 & 1.41E-06 & 1.09  & \multirow{2}[2]{*}{37.5} \bigstrut[t]\\
		ubMCQMC-L & 82    & 32    & 120 & 4.65E-04 & 0.92  &       &       & 50    & 16    & 73 & 1.04E-05 & 0.89 & \bigstrut[b]\\
		\hline
	\end{tabular}%
}
\label{tab:ineffprobit}%
\end{table}%


\subsection{Bayesian logistic regression model} \label{subsec:logistic}
Last, we consider a more challenging problem, the Bayesian logistic regression model with P\'olya-Gamma (PG) Gibbs sampler \citep{polson2013} which was used in \citep[Section S4]{jacob2020}. Given the observations of binary response variable $\bm{y} = (y_1,y_2,\ldots,y_n) \in \{0,1\}^{n\times 1}$, the logistic model is expressed as $\mathbb{P}(y_i = 1) = 1/(1+\exp(-\bm{D}_i\bm{\beta}))$. Assume that the prior distribution of $\bm{\beta}$ is $\bm{\beta} \sim \mathcal{N}(\bm{b}, \bm{B})$, where $\bm{b} = \bm{0}^{p \times 1}$ and $\bm{B} = \text{diag}(10,10,\ldots,10)$ is a $p\times p$ diagonal matrix. Our goal is to estimate the posterior mean of \(\bm{\beta}\). The PG Gibbs sampler is extended with $n$ PG distribution variables $\{W_i\}_{i=1}^{n}$ targeting the posterior distribution. The PG distribution with parameters $(1,c)$ is denoted by PG$(1,c)$.
Denote $\bm{W} = (W_1,\ldots,W_n)$ and  $\tilde{\bm{y}}=\left(y_1-1 / 2, \ldots, y_n-1 / 2\right)$. The update process of the PG Gibbs sampler is given by 
$$
\begin{aligned}
\bm{\beta} | W_1,W_2,\ldots,W_n & \sim \mathcal{N}\left(\Sigma(\bm{W})\left(\bm{D}^T \tilde{\bm{y}}+\bm{B}^{-1} \bm{b}\right), \Sigma(\bm{W})\right), \quad  \\
W_i | \bm{\beta},W_{-i} & \sim \mathrm{PG}\left(1,\left|\bm{D}_i \bm{\beta}\right|\right) \quad i = 1, \ldots, n,
\end{aligned}
$$
where  $\Sigma(\bm{W})=\left(\bm{D}^T \operatorname{diag}(\bm{W}) \bm{D}+\bm{B}^{-1}\right)^{-1}$. The probability density function $f_{\mathrm{pg}}(x;1,c)$ of PG$(1,c)$ is intractable, but the ratio of two density evaluations can be calculated by 
$$
\quad \frac{f_{\mathrm{pg}}\left(x ;1,c_2\right)}{f_{\mathrm{pg}}\left(x ;1,c_1\right)}=\frac{\cosh \left(c_2 / 2\right)}{\cosh \left(c_1 / 2\right)} \exp \left(-\left(\frac{c_2^2}{2}-\frac{c_1^2}{2}\right) x\right),\ x>0.
$$
Hence, the update of $W_i$ can also be implemented fast in the maximal coupling method.

Polson et al. \citep{polson2013} gave an acceptance-rejection procedure to sample PG$(1,c)$ distribution, in which a mixture of a truncated inverse Gaussian distribution and a truncated exponential distribution is used as the proposal.  The PG Gibbs sampler thus cannot avoid the acceptance-rejection process because of the sampling of PG$(1,c)$ distributions, which may degrade the performance of ubMCQMC. So this sampler is more challenging than the last two samplers for ubMCQMC. Throughout the entire acceptance-rejection process, we use array-WCUD points only for the initial sampling the mixture proposal. If the proposal is rejected, we continue to use IID uniform points in the subsequent steps. 

For ubMCMC, we follow the sampling scheme for PG distribution in \citep[Appendix 1]{windle2013}. For ubMCQMC, we know that
sampling from a mixture distribution typically involves using a random variable to determine which distribution to sample from, followed by sampling from the chosen distribution. In the following, we consider three difference sampling schemes for the mixture proposal in ubMCQMC.\\
\textbf{Approach 1:} We use an IID $U(0,1)$ variable to determine a component of the mixture distribution, and then using an array-WCUD point to sample the chosen component. If the chosen component is the truncated inverse Gaussian distribution, we use the acceptance-rejection sampling method in \citep[Appendix 1]{windle2013}, in which the array-WCUD point is only used in the initial sampling. For a fixed array-WCUD point, the output for the  mixture distribution is thus not unique, potentially disrupting the uniformity of the samples. \\
\textbf{Approach 2:} We replace the IID $U(0,1)$ variable in Approach 1 by an array-WCUD point. This method slightly improves the uniformity of the samples, but the number of array-WCUD points used in each iteration is increased from $p+n$ to $p+2n$, leading to a higher dimensional problem for ubMCQMC.\\
\textbf{Approach 3:} We sample the mixture distribution via its inverse CDF. Clearly, the inverse CDF is not available in a closed form, and we use Newton's method to approximate it. This calls for an extra computation to estimate the inverse CDF. Compared to Approaches 1 and 2,  this approach can avoid introducing the discontinuities in the update function from the components selection and the sampling of the truncated inverse Gaussian distribution, while keeping the number of the array-WCUD points used in each iteration as the dimension of states $d=p+n$.

In Table \ref{tab:diffm}, we compare ubMCQMC-H combined with the three approaches to ubMCMC for the Pima Indian dataset \citep{polson2013,windle2013} with $n = 392$ observations on $9$ predictor variables, resulting in $d = 401$, to illustrate the performance of the three approaches, in which we choose $k = 32$ and $ N \in \{2^{10},2^{13},2^{16}\}$. Although this example has high dimensionality and involves the acceptance-rejection process, ubMCQMC methods still achieve some reduction in RMSE, especially for Approach 3. We find that the RMSE reduction factors are approximately $2$ for Approach 1, which are insensitive with respect to \(N\). Approach 2 gains larger reduction factors but the number of array-WCUD points used in each iteration increases from $401$ to  $793$. In contrast, Approach 3 maintains the number of  array-WCUD points used in each iteration at $401$ and achieves the largest reduction factor, increasing from $9$ to $22.6$ when the sample size goes up. Although the runtime of Approach 3 is approximately twice that of ubMCMC, its benefits are substantial due to the use of the inverse CDF maintaining the balance of array-WCUD points. Taking $N=2^{10}$ for example, if we want to reduce RMSE by a factor of $9$ in ubMCMC, the sample size needs to increase by $81$ times, which will roughly increase the runtime by a factor of  $81$.

\begin{table}[htbp]
\centering
\caption{Total RMSE reduction factors and runtimes for  the Pima dataset.}
\scalebox{0.85}{
	\begin{tabular}{ccccccccccccc}
		\hline
		\multirow{3}[4]{*}{$N$} & & & & \multicolumn{9}{c}{ubMCQMC-H}\\\cline{6-13}
		& &\multicolumn{2}{c}{ubMCMC} & &
		\multicolumn{2}{c}{Approach 1} & & \multicolumn{2}{c}{Approach 2} & & \multicolumn{2}{c}{Approach 3} \bigstrut\\
		\cline{3-4}\cline{6-7} \cline{9-10} \cline{12-13}& &  RMSE & Time (s) & & RRF & Time (s)& &  RRF & Time (s) & & RRF & Time (s) \bigstrut\\
		\hline
		$2^{10}$ & & 2.33E-03 & 0.50 & & 2.11 & 0.60 &  & 4.64 & 0.62 &   & 9.00  & 0.96 \bigstrut[t]\\
		$2^{13}$ & & 7.99E-04 & 3.73 & & 2.05 & 4.47  & & 4.86 & 4.66 &  & 18.65  & 7.00 \\
		$2^{16}$ & & 2.76E-04 & 29.77 & & 2.14 & 35.72 & & 5.57 & 37.35 & & 22.60  & 55.72\bigstrut[b]\\
		\hline
	\end{tabular}%
}
\label{tab:diffm}%
\end{table}


We also examine Approach 3 for a larger dataset --- the German credit dataset with $1000$ observations and $20$ predictor variables, which was used in \citep{jacob2020,polson2013}.  Among these $20$ variables, seven are quantitative, while the rest are categorical. After converting the categorical variables into binary variables, we obtain \(49\) predictor variables, resulting in $d = 1049$. 

For this dataset, we choose $k=260$. From Figure \ref{fig:logistic_tvar}, the convergence rate of ubMCQMC-H for $d=1049$ is slightly better than $O(N^{-1/2})$. Comparing to the $O(N^{-1})$ convergence rate of ubMCQMC-H observed in the Bayesian linear regression model, ubMCQMC clearly suffers from the curse of dimensionality. However, for such a challenging  problem, ubMCQMC still performs better than ubMCMC for all cases in Figure \ref{fig:logistic_tvar}. Total RMSE reduction factors of ubMCQMC-H with different values of \(N\) for the German dataset are also provided in Table \ref{tab:logistic}. It is observed that for this sampler with 1049 dimensions, ubMCQMC-H still reduces the total RMSE of ubMCMC by a factor ranging from $6.06$ to $11.79$.

\begin{figure}[htbp]
\vspace{-4mm}
\centering
\subfigure[Total RMSEs for Pima.]{\label{logistica}
	\begin{minipage}[t]{0.4\linewidth}
		\centering
		\includegraphics[width=0.9\linewidth]{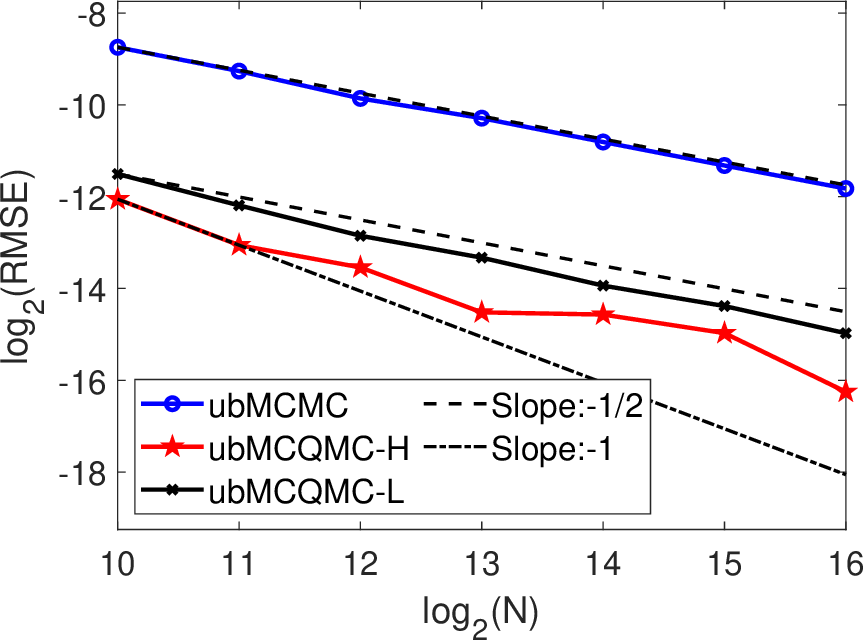}
	\end{minipage}
}
\vspace{-3mm}
\subfigure[Total RMSEs for German.]{\label{logisticb}
	\begin{minipage}[t]{0.4\linewidth}
		\centering
		\includegraphics[width=0.9\linewidth]{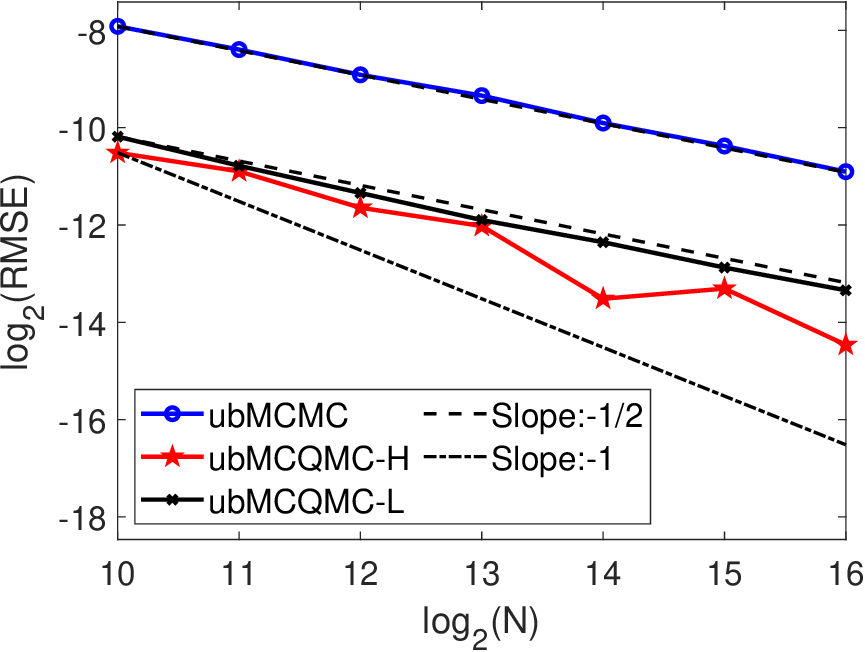}
	\end{minipage}
}
\\
\subfigure[Component-wise RMSEs for Pima.]{\label{logisticc}
	\begin{minipage}[t]{0.4\linewidth}
		\centering
		\includegraphics[width=0.9\linewidth]{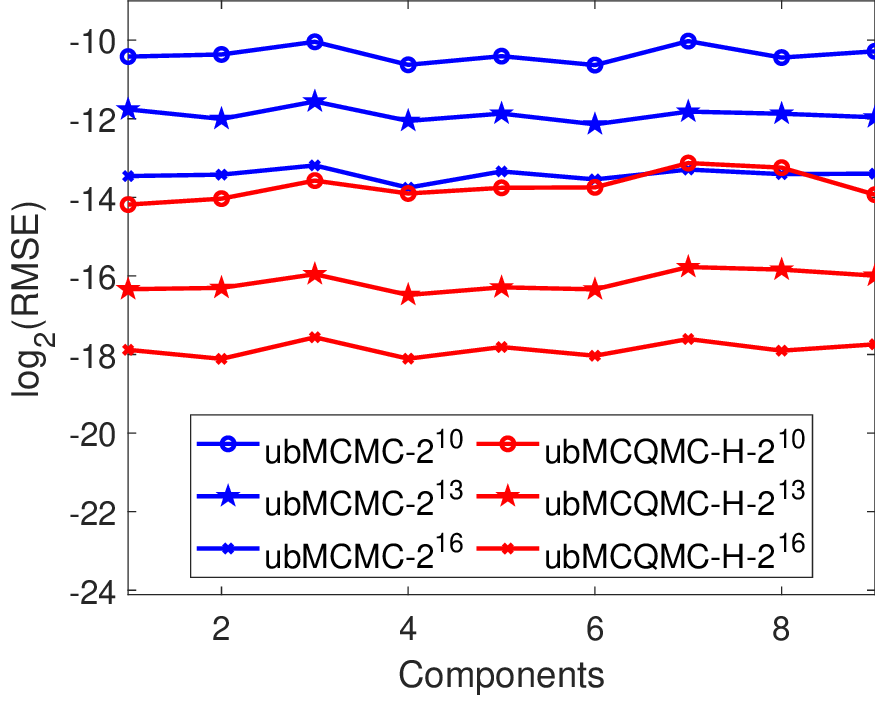}
	\end{minipage}
}
\vspace{-3mm}
\subfigure[Component-wise RMSEs for German.]{\label{logisticd}
	\begin{minipage}[t]{0.4\linewidth}
		\centering
		\includegraphics[width=0.9\linewidth]{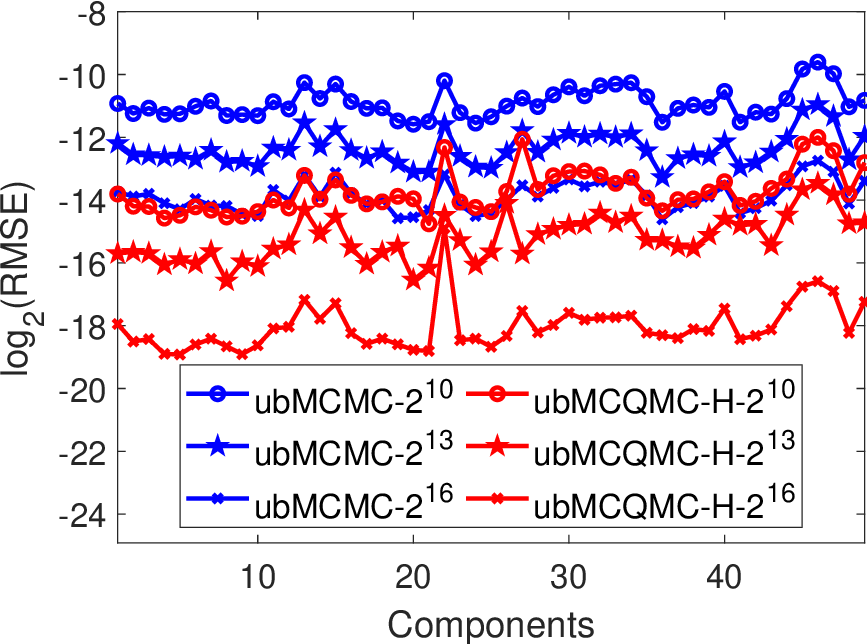}
	\end{minipage}
}
\caption{RMSEs for the Pima and German datasets. Unbiased MCQMC methods are combined with Approach 3.}
\label{fig:logistic_tvar}
\vspace{-6mm}
\end{figure}

\begin{table}[htbp]
\centering
\caption{Total RMSE reduction factors and runtimes for the German dataset.}
\scalebox{0.85}{
	\begin{tabular}{ccccccc}
		\hline
		\multirow{2}[6]{*}{$N$}  &     & \multicolumn{2}{c}{ubMCMC} &    & \multicolumn{2}{c}{Approach 3} \bigstrut\\
		\cline{3-4}\cline{6-7}       &       & RMSE  & Time (s)  &       & RRF  & Time (s)  \bigstrut\\
		\cline{1-7}    
		$2^{10}$     &       & 4.14E-03 & 4.87 &       & 6.06  & 8.85     \bigstrut[t]\\
		$2^{13}$     &       & 1.54E-03 & 27.95 &       & 6.41  & 42.41    \\
		$2^{16}$     &       & 5.22E-04 & 183.38 &       & 11.79 & 260.78  \bigstrut[b]\\
		\hline
	\end{tabular}%
}
\vspace{-4mm}
\label{tab:logistic}%
\end{table}

We then take a small sample size $N$ to compare the efficiency of $F_{k:m}(\black{\mathbf{X}},\black{\mathbf{Y}})$ and that of MCMC estimators. We only consider Approach 3 with Liao's method, which performed better than Approach 1 and Approach 2, but consumed nearly twice runtime of ubMCMC, see Tables \ref{tab:diffm} and \ref{tab:logistic}. For a fair comparison, we set $c = 2c_0$ in \eqref{eq:efficiency}. Following the setting in Bayesian linear regression model, the expected cost, the total variance, and the loss of efficiency of ubMCMC and ubMCQMC are presented in Table \ref{tab:inefflogistic}. We observe that for the Pima dataset, the efficiency of ubMCMC with $N = 640$ and the efficiency of ubMCQMC with $N = 26$ are comparable to the efficiency of standard MCMC, implying that using ubMCQMC saves a sample size of $640/26 = 24.6$ times. However, for the case of $d=1049$ in using German dataset, the sample size saving factor is $60.9$. This is because when $k$ and $E[\tau]$ are large, the gap between expected cost and $N$ becomes significant. Furthermore, due to the slow convergence rate of ubMCMC, it requires a large $N$ to make the loss efficiency approach $1$. In contrast, unbiased MCQMC has a faster convergence rate, only a small $N$ is needed to achieve the efficiency of classical MCMC. Consequently, the sample size saving factor becomes significant under these circumstances. 

\begin{table}[htbp]
\vspace{-4mm}
\centering
\caption{The loss of efficiency for the Pima and German datasets.}
\scalebox{0.85}{
	\begin{tabular}{cccccccccccccc}
		\hline
		\multirow{2}[4]{*}{} & \multicolumn{6}{c}{Pima}                    &       & \multicolumn{6}{c}{German} \bigstrut\\
		\cline{2-7}  \cline{9-14}       &   $k$  & $N$   & Cost  & Variance   & Loeff & SF &  &   $k$  & $N$   & Cost  & Variance  & Loeff & SF  \bigstrut\\
		\hline
		ubMCMC & 32    & 640   & 678  & 7.92E-07 & 1.06  & \multirow{2}[2]{*}{24.6} &  & 260   & 3900  & 4195 & 4.67E-07 & 1.07  & \multirow{2}[2]{*}{60.9} \bigstrut[t]\\
		ubMCQMC-L & 32    & 26    & 127 & 3.96E-06 & 0.96  &       &       
		& 260   & 64    & 719 & 2.75E-06 & 1.08 & \bigstrut[b]\\
		\hline
	\end{tabular}%
}
\label{tab:inefflogistic}%
\end{table}%

\section{Conclusions and discussions}
\label{sec:conclusions}
By incorporating array-WCUD sequences into the first-step sampling of the maximal coupling method in unbiased MCMC, we propose unbiased MCQMC method for Gibbs samplers. This method maintains the unbiasedness of the estimator and potentially accelerates the convergence rate of unbiased MCMC.  Moreover, without relying on Markov property, which was required in \citep{jacob2020}, we provide a simpler proof for the upper bound of the second moment of the bias correction term. This pays the way to study the RMSE rate of unbiased MCQMC if the error rate of MCQMC is well established. Through a series of comprehensive experiments, unbiased MCQMC has shown significant improvement in RMSE reduction compared to unbiased MCMC. Especially in low dimensional cases, compared to the $O(N^{-1/2})$ convergence rate of unbiased MCMC, unbiased MCQMC achieves an almost $O(N^{-1})$ convergence rate. In high dimensional cases, unbiased MCQMC suffers from the curse of dimensionality as plain QMC. Surprisingly, for the Bayesian logistic regression model  with P\'olya Gamma Gibbs sampler in $1049$ dimensions, unbiased MCQMC also can reduce the RMSE dozens of times. In the setting of parallelization, unbiased MCQMC can also save dozens of times the sample size to achieve a similar efficiency of unbiased MCMC.

There are some limitations to unbiased MCQMC methods. Firstly, unbiased MCQMC is amenable to smooth update functions in \eqref{gibbs}. For M-H samplers that involve an acceptance-rejection process or any sampling process that includes such a mechanism, the uniformity of the resulting samples will be destroyed, suppressing the effectiveness of unbiased MCQMC. On the other hand, it is necessary to explicitly express the transition kernel in the form \(\bm{X}_t = \phi(\bm{X}_{t-1}, \bm{u})\), where uniform variables are used to advance the chain. To this end, special care should be taken for complex sampling processes. Secondly, as observed in our experiments, unbiased MCQMC suffers from the curse of dimensionality. Lastly, one needs to choose an appropriate burn-in period $k$ for unbiased MCQMC. The value of $k$ used in unbiased MCMC and MCQMC methods should be different. We prefer to choose a larger value of $k$ in unbiased MCQMC. In our numerical experiments, we chose twice the 99\% sample quantile of \(\tau\) rather than the 99\% sample quantile  used in \citep{jacob2020}. This strategy worked well in our numerical examples. However, the factor of two is an empirical choice. 
Our numerical examples exhibits short meeting times, in which the $k$ was set between 6 and 260. For some extreme settings, such as the bad initial distribution or ineffective coupling strategy, the meeting time could be large \citep[Section 6]{jacob2020}. It is unclear how to choose a good value of $k$ for such cases. 
However, it should be noted that, regardless of the size of the dimension and the choice of \(k\) (even if \(k=1\)), unbiased MCQMC consistently outperforms unbiased MCMC as shown in both the main text and Appendix \ref{app:k}, respectively. Addressing or mitigating these limitations to better leverage the performance of unbiased MCQMC will be the directions for our future work.






\bibliographystyle{plain}
\bibliography{newreferences}

\newpage
\appendix

The following Appendix contains the introducion of Liao's method and Harase's method, the implementation of unbiased MCQMC algorithm, and numerical experiments on the impact of bias correction term and the burn-in period $k$.  

\section{An introduction of Liao's method and Harase's method}\label{app:method}
We know that using typical low-discrepancy sequences directly in MCMC will lead to incorrect results due to strong correlations between the sequences. Liao's method is to randomize the order of the given $d$-dimensional low-discrepancy sequences such as Sobol' sequence. Mathematically, it can be described as follows: Let $N$ be the sample size, consider $d$-dimensional low-discrepancy sequences $\boldsymbol{u}_1,\boldsymbol{u}_2,\ldots,\boldsymbol{u}_N \in (0,1)^d$ and a random permutation $\bm\zeta$ of the set $\{1,2,\ldots,N\}$. Define $v^{(L)}_{((i-1)d+j)} = \boldsymbol{u}_{\zeta(i),j}$, where $\boldsymbol{u}_{\zeta(i),j}$ represents the $j$-th coordinate of $\boldsymbol{u}_{\zeta(i)}$, $i=1,2,\ldots,N$, $j=1,2,\ldots,d$. Then, $\{v_i^{(L)}\}_{i=1}^{Nd}$ is an array-WCUD sequence of length $Nd$ \citep{tribble2008}. We refer to Figure 1 in \citep{tribble2008} for a more intuitive understanding of the application of Liao's method. 
The variable matrix generated by Liao's method is 
\begin{equation}\label{eq:vl}
	V_L = \left[\begin{array}{cccc}
		v^{(L)}_1 & v^{(L)}_2 & \cdots & v^{(L)}_{d} \\
		v^{(L)}_{d+1} & v^{(L)}_{d+2} & \cdots & v^{(L)}_{2d} \\
		\vdots & \vdots & \ddots & \vdots \\
		v^{(L)}_{(N-1)d+1} & v^{(L)}_{(N-1) d+2} & \cdots & v^{(L)}_{Nd}
	\end{array}\right] = \left[\begin{array}{c}
		\boldsymbol{u}_{\zeta(1)} \\
		\boldsymbol{u}_{\zeta(2)}\\
		\vdots  \\
		\boldsymbol{u}_{\zeta(N)}
	\end{array}\right].
\end{equation}
Following \cite{tribble2008}, we apply the random shifts algorithm \citep{cranley1976} given in Algorithm \ref{alg:cp} to randomize the matrix $V_L$. Denote the randomized matrix as $V_L^*$. By Lemma 2 in \cite{tribble2008}, the driving sequence generated by $V_L^*$ is also an array-WCUD sequence.

While Liao's method may not outperform other methods at specific sample sizes in practice, it does not have any requirements on the sample size. Therefore, in our numerical experiments, we utilize Liao's method to achieve this flexibility.
\begin{algorithm}[htbp]
	\caption{The random shifts algorithm}
	\label{alg:cp}
	\begin{algorithmic}[1]
		\STATE{Given a $N\times d$ matrix $A$ with entries $a_{ij}$.}
		\STATE{For each column $1\le j\le d$, draw $z_j\sim U[0,1]$ independently.}
		\STATE{For the $(i,j)$-th entry, set  $a^*_{i,j} = a_{i,j}+z_j\text{~mod~} 1$.}
		\RETURN Matrix $A^*$ with entries $a^*_{ij}$.
	\end{algorithmic}
\end{algorithm}

Using CUD sequences in MCMC is akin to using the entire period of a random number generator with a small period \cite{niederreiter1986}. We also use the LFSR generator constructed with the optimization of the t-value by Harase in specific sample sizes \cite{harase2021}.  Harase \cite{harase2021} emphasized that LFSR generators essentially possess a digital net structure, and for sequences with digital net structures, the t-value serves as a critical optimization metric. 

Following \citep[Section 8.1]{chen}, LFSR generator constructs its variates from a sequence $\{b_i\}_{i=0}^{\infty} \in \{0,1\}$. Given a positive integer $n$ and initial states $b_0,b_1,\ldots,b_{n-1}$ that are not entirely composed of zeros and some integers $a_0,a_1,\ldots,a_{n-1}$ of zeros and ones, LFSR generator of order $n$ updates $b_i,i\ge n$ by the recursive formula
$$
b_i = \sum_{j=0}^{n-1} a_j b_{i-n+j} ~~ \text{mod~} 2.
$$
Obviously, the $n-$tuple $(b_i,b_{i+1},\ldots,b_{i+n-1})$ has a period at most $N-1$ with  $N= 2^{n}$. The period is exactly equal to $N-1$ if and only if the characteristic polynomial $p(x) = \sum_{j=0}^{n-1} a_jx^j + x^n$ is a primitive polynomial over the Galois field with two elements \cite{booknieder1992}. Given the sequence $\{b_i\}_{i\ge0}$ and an offset (or step size) $g>0$ such that gcd$(g,N-1) = 1$, $v^{(H)}_i$ is constructed with 
\begin{equation*}
	v^{(H)}_i=\sum_{j=0}^{w-1} b_{(i-1)g+j} 2^{-j-1}, \quad i=1,2,\ldots,
\end{equation*}
where $w$ is a digit number and $\{v^{(H)}_i\}_{i=1}^{\infty}$ have a period of $N-1$. Harase proposed $w=32$ and provided a table of $g$ and $(a_0,a_1,\ldots,a_{n-1})$ for $10\le n\le32$ in Table 1 of \cite{harase2021}. If $\gcd(d,N-1) = 1$, we can simply define a $N\times d$ variable matrix by the following $N$ points, 
\begin{equation*}
	(0,\ldots,0),(v^{(H)}_1, \cdots, v^{(H)}_{d}),(v^{(H)}_{d+1}, \cdots, v^{(H)}_{2d}),\cdots,(v^{(H)}_{(N-2)d+1},\cdots, v^{(H)}_{(N-1)d}).
\end{equation*}
If $\gcd(d,N-1) = m > 1$, then we generate $m$ distince short loops as follows,
\begin{equation*}
	(v^{(H)}_j, \cdots, v^{(H)}_{j+d-1}),(v^{(H)}_{j+d}, \cdots, v^{(H)}_{j+2d-1}),\cdots,(v^{(H)}_{j+((N-1)/m)-1)d},\cdots, v^{(H)}_{j+((N-1)/m)d-1}),
\end{equation*}
for $j = 1,\ldots,m$. The $N\times d$ variable matrix is defined by concatenating $(0,\ldots,0)$ with them in order. 

Denote the variable matrix generated by Harase's method as $V_H$. Following \cite{chen,harase2021,tribble}, we adopt the digital shifts algorithm \citep{owenqmc} given in Algorithm \ref{alg:ds} to randomize the matrix $V_H$, and following \cite{harase2021}, we do the digital shift up to the $32$-nd digit in Algorithm \ref{alg:ds}. Denote $V_H^*$ as the resulting matrix after the randomization. By Theorem 5.2.2 in \cite{tribble}, the driving sequence generated by $V_H^*$  is an array-WCUD sequence. 

\begin{algorithm}[htbp]
	\caption{The digital shifts method}
	\label{alg:ds}
	\begin{algorithmic}[1]
		\STATE{Given a $N\times d$ matrix $A$ with entries $a_{ij}$.}
		\STATE{For each column $1\le j\le d$, draw $z_j\sim U[0,1]$ independently.}
		\STATE{Write $z_j$'s binary expansion as $0.z_j^{(1)}z_j^{(2)}\ldots $.}
		\STATE{For the $(i,j)$-th entry $a_{i,j}$, write its binary expansion as $0.a_{i,j}^{(1)}a_{i,j}^{(2)}\ldots$.}
		\STATE{Set $b_{i,j}^{(l)} = a_{i,j}^{(l)} + z_j^{(l)} \text{~mod~} 2$ and $a^*_{i,j} = \sum\limits_{l=1}^{\infty}b_{i,j}^{(l)}2^{-l}$.}
		\RETURN Matrix $A^*$ with entries $a^*_{i,j}$.
	\end{algorithmic}
\end{algorithm}

\section{Algorithm Implementation} \label{app:alg}
In summary, our unbiased MCQMC algorithm can be described in Algorithm \ref{alg:ubmcqmc} as follows. 
\begin{algorithm}[H]
	\caption{Unbiased MCQMC algorithm}
	\label{alg:ubmcqmc}
	\begin{algorithmic}[1]
		\STATE{Given integers $k$, sample size $N$, $m = N+k-1$, initial distribution $\pi_0$, update function $\phi$ and generator function $\psi_i$ defined in \eqref{gibbs}, conditional density $P_i,i=1,2,\ldots,s$ of the components of each part with $d = \sum_{i=1}^s d_i$.}
		\STATE{Construct a $N\times d$ variable matrix $V$ by Harase's method or Liao's method}
		\STATE{Set $V^*$ as the matrix $V$ after randomization}
		\STATE{Construct $(k-1)\times d$ random matrix $V_R$ generated by IID sequence.}
		\STATE{Set $m \times d$ matrix $V^{(\mathbf{X})} = [V_R;V^*]$}
		\STATE{Sample ${\bm{X}_0},\black{{\bm{Y}_0}}\sim \pi_0$, and set $\bm{u}_1$ as the first row of $V^{(\mathbf{X})}$, ${\bm{X}_{1}} = \phi({\bm{X}_0},\bm{u}_1)$, $t=1,\tau=\infty$}
		\STATE{Repeat Steps 8-22 until $t \ge \max(m,\tau)$}
		\IF{$t < m$}
		\STATE{Set $\bm{u}_{t+1}$ as the $(t+1)$-th row of $V^{(\mathbf{X})}$}
		\ELSE \STATE{Sample $\bm{u}_{t+1} \sim U(0,1)^d$}
		\ENDIF
		\IF{$t < \tau$}
		\STATE{Set $\black{({\bm{x}_1}, \ldots,{\bm{x}_s}) = {\bm{X}_t}, ({\bm{y}_1}, \ldots,{\bm{y}_s}) = {\bm{Y}_{t-1}}}$, $({\bm{v}}_1, \ldots,\bm{v}_s) = \bm{u}_{t+1}$}
		\FOR{$i$ in $1$ to $s$}
		\STATE{Set $p_i(\cdot) = P_i(\black{{\bm{x}_{-i}}},\cdot),q_i(\cdot) = P_i(\black{{\bm{y}_{-i}}},\cdot)$}
		\STATE{Sample $w \sim \black{\mathbf{U}}(0,1)$, set $\black{{\bm{x}_{i}}} = \psi_i(\black{{\bm{x}_{-i}}},\bm{v}_i)$, if $p_i(\black{{\bm{x}_{i}}})w \le q_i(\black{{\bm{x}_{i}}})$, set $\black{{\bm{y}_{i}}} = \black{{\bm{x}_{i}}}$, otherwise, sample $\bm{v} \sim \black{\mathbf{U}}(0,1)^{d_i}$ and $w' \sim \black{\mathbf{U}}(0,1)$, set $\black{{\bm{y}_{i}}} = \psi_i(\black{{\bm{y}_{-i}}},\bm{v})$ until $q_i(\black{{\bm{y}_{i}}})w' > p_i(\black{{\bm{y}_{i}}})$}
		\ENDFOR
		\STATE{Set $\black{{\bm{X}_{t+1}} = ({\bm{x}_1}, \ldots,{\bm{x}_s}), \black{{\bm{Y}_{t}}} = ({\bm{y}_1}, \ldots,{\bm{y}_s})}$. If $\black{{\bm{X}_{t+1}} = {\bm{Y}_{t}}}$, set $\tau = t$ and $t = t+1$} 
		\ELSE
		\STATE{Set $\black{{\bm{X}_{t+1}} = \phi({\bm{X}_t},\bm{u}_{t+1}), {\bm{Y}_t} = {\bm{X}_{t+1}}}$ and $t = t+1$}
		\ENDIF
		\STATE{Return $F_{k:m}(\black{\mathbf{X}},\black{\mathbf{Y}})$ by \eqref{eq:Fkm}}
	\end{algorithmic}
\end{algorithm}

\section{Numerical experiments on the impact of the bias correction term and the burn-in period}\label{app:k}
In Section \ref{sec:experiments}, we used array-WCUD sequences after the $(k-1)$-th step, where \( k=2\bar k \) and $\bar k$ is the 99\% sample quantile of \( \tau \) based on 1000 samples of the meeting time $\tau$. In this section, we look at the impact of $k$ and the effect of using array-WCUD sequence in the burn-in period.
To study the effect of $k$, we take \(k \in\{ 1,\ \bar k,\ 2\bar k,\ 4\bar k\}\). To investigate the necessity of using array-WCUD sequences in the burn-in period, in the following tables, ``Case 1" denotes using WCUD sequences throughout including the burn-in period, while ``Case 2" uses IID sequences instead in the burn-in period followed by WCUD sequences subsequently. 

We repeat the experiment $25$ times with $R=100$ and $N=2^{10}$ to obtain 25 total RMSEs and then take the average as the final total RMSE. Moreover, we show the coefficient of variation (CV) for the $25$ total RMSEs to examine the stability, which is defined by the standard deviation of the $25$ total RMSEs divided by the final total RMSE. 

Table \ref{tab:boston1} shows the results for the Bayesian linear regression model with the Boston dataset. We can see that for all cases, ubMCQMC consistently outperforms  ubMCMC, especially for Harase's method in Case 2. 
These results favor the necessity of using array-WCUD sequence after the burn-in period.
On the other hand, the value of $k$ has an impact on the effectiveness of ubMCQMC. If we take a small $k$, the impact of the bias correction term could be significant for ubMCQMC. This motivates us to take a large value of $k$. But, an excessively large $k$ only results in sample waste for ubMCQMC. Determining an appropriate $k$ in ubMCQMC is a worthwhile consideration for future work. The choice of $2\bar{k}$ is merely empirical in our work. However, despite the greater influence of the bias correction term on ubMCQMC compared to ubMCMC, for any choice of \(k\) (even if \(k=1\)), ubMCQMC consistently outperforms ubMCMC.

\begin{table}[htbp]
	\centering
	\caption{The total RMSE reduction factors and CVs for the Boston dataset.}
	\begin{tabular}{cccccccccc}
		\hline
		\multirow{2}[2]{*}{$N = 2^{10}$}   & \multirow{2}[2]{*}{$k$} & \multicolumn{2}{c}{ubMCMC} & & \multicolumn{2}{c}{ubMCQMC-H} & & \multicolumn{2}{c}{ubMCQMC-L}\\
		\cline{3-4}\cline{6-7}\cline{9-10}
		&      & RMSE & CV & & RRF  & CV & & RRF  & CV \bigstrut\\
		\hline
		\multirow{4}[2]{*}{Case 1} 
		& 1     & 6.46E-04 & 0.06 & & 3.94  & 0.30  & & 3.81  & 0.36 \bigstrut[t]\\
		& 4     & 6.18E-04 & 0.05 &  & 15.26 & 0.98  & & 9.87  & 0.29 \\
		& 8     & 6.22E-04 & 0.05 &  & 12.16 &  0.02 & & 8.84  & 0.06 \\
		& 16    & 6.22E-04 & 0.06 & & 9.84  & 0.03 & & 6.97  & 0.05 \bigstrut[b]\\
		\hline
		\multirow{4}[2]{*}{Case 2} 
		& 1   &  6.42E-04 & 0.06 & & 3.58 & 0.37 & & 3.44 & 0.37 \bigstrut[t]\\
		& 4   &      6.24E-04 & 0.06 & & 22.53 & 1.65 & & 13.12 & 0.04 \\
		& 8  &   6.23E-04 & 0.04 &  & \textbf{75.83} & \textbf{0.07} & & 12.96 & 0.05 \\
		& 16  &      6.19E-04 & 0.05 &  & 76.01 & 0.07 &  & 12.73 & 0.05 \bigstrut[b]\\
		\hline
	\end{tabular}%
	\label{tab:boston1}%
\end{table}%

We next present the results of the Bayesian probit regression model with the Vaso dataset and the Bayesian logistic regression model with the Pima dataset in Tables \ref{tab:vaso1}-\ref{tab:pima1}, respectively. Similar phenomenons can be observed in these cases. 

\begin{table}[H]
	\centering
	\caption{The total RMSE reduction factors and CVs for the Vaso dataset.}
	\begin{tabular}{cccccccccc}
		\hline
		\multirow{2}[2]{*}{$N = 2^{10}$}   & \multirow{2}[2]{*}{$k$} & \multicolumn{2}{c}{ubMCMC} & & \multicolumn{2}{c}{ubMCQMC-H} & & \multicolumn{2}{c}{ubMCQMC-L}\\
		\cline{3-4}\cline{6-7}\cline{9-10}
		&      & RMSE & CV & & RRF  & CV & & RRF  & CV \bigstrut\\
		\hline
		\multirow{4}[2]{*}{Case 1}   
		& 1   &  4.04E-02 & 0.40 & & 1.27  & 0.63 &  & 1.19  & 0.66 \bigstrut[t]\\
		& 41  &  2.64E-02 & 0.24 & & 4.43  & 0.21 & & 2.81  & 1.10 \\
		& 82  &  2.54E-02 & 0.12 & & 3.35  & 0.06 & & 2.97  & 0.09 \\
		& 164 &  2.60E-02 & 0.07 &  & 2.54  & 0.07 & & 2.28  & 0.08 \bigstrut[b]\\
		\hline
		\multirow{4}[2]{*}{Case 2} 
		& 1  & 4.69E-02 & 0.41 & & 1.87  & 0.55 &  & 1.27  & 0.73 \bigstrut[t]\\
		& 41 & 2.77E-02 & 0.33 &  & 5.78  & 0.39 & & 3.66  & 0.64 \\
		& 82 & 2.46E-02 & 0.07 & & \textbf{6.43} & \textbf{0.08} & & 4.31  & 0.06 \\
		& 164& 2.39E-02 & 0.07 & & 6.12  & 0.08 &  & 4.10  & 0.07 \bigstrut[b]\\
		\hline
	\end{tabular}%
	\label{tab:vaso1}%
\end{table}%

\begin{table}[H]
	\centering
	\caption{The total RMSE reduction factors and CVs for the Pima dataset.}
	\begin{tabular}{cccccccccc}
		\hline
		\multirow{2}[2]{*}{$N = 2^{10}$}   & \multirow{2}[2]{*}{$k$} & \multicolumn{2}{c}{ubMCMC} & & \multicolumn{2}{c}{ubMCQMC-H} & & \multicolumn{2}{c}{ubMCQMC-L}\\
		\cline{3-4}\cline{6-7}\cline{9-10}
		&      & RMSE & CV & & RRF  & CV & & RRF  & CV \bigstrut\\
		\hline
		\multirow{4}[2]{*}{Case 1}   
		& 1     & 3.24E-03 & 0.09 &  & 1.19  & 0.17 & & 1.37  & 0.17 \bigstrut[t]\\
		& 16    & 2.69E-03 & 0.26 &  & 1.18  & 0.45 & & 2.30  & 0.89 \\
		& 32    & 2.28E-03 & 0.03 & & 4.30  & 0.02 & & 4.48  & 0.02 \\
		& 64    & 2.32E-03 & 0.04 & & 3.61  & 0.03 &  & 3.55  & 0.02 \bigstrut[b]\\
		\hline
		\multirow{4}[2]{*}{Case 2} 
		& 1   & 3.49E-03 & 0.14 & & 1.30  & 0.21 & & 1.45  & 0.18 \bigstrut[t]\\
		& 16  & 2.60E-03 & 0.27 & & 2.95  & 1.18 & & 2.68  & 1.11 \\
		& 32  & 2.27E-03 & 0.04 & & \textbf{9.17} & \textbf{0.03} & & 6.86  & 0.02 \\
		& 64  & 2.25E-03 & 0.03 & & 9.15  & 0.03 & & 6.78  & 0.03 \bigstrut[b]\\
		\hline
	\end{tabular}%
	\label{tab:pima1}%
\end{table}%

\end{document}